\documentclass[11pt,a4paper]{article}
\usepackage{amsmath}
\usepackage[a4paper]{geometry}
\usepackage{amssymb,latexsym,amsmath,amsfonts,amsthm}
\usepackage{graphicx}
\usepackage{algorithm}
\usepackage{algorithmic}
 \usepackage{dsfont}
\usepackage{mathrsfs}
\usepackage{epsfig}
\usepackage{booktabs}
\usepackage{mathtools}
\usepackage{comment,verbatim}
\usepackage{overpic}
\usepackage{hyperref}
\usepackage{bbm}
\usepackage[toc,page]{appendix}

\renewcommand{\vec}{\mathbf}
\newcommand{\ud}{\,\mathrm{d}}

\newtheorem{theorem}{Theorem}[section]
\newtheorem{lemma}[theorem]{Lemma}
\newtheorem{proposition}[theorem]{Proposition}

\newtheorem{corollary}[theorem]{Corollary}

\theoremstyle{definition}
\newtheorem{definition}[theorem]{Definition}

\newtheorem{assum}[part]{Assumption}

\theoremstyle{remark}

\newtheorem{remark}[theorem]{Remark}

\numberwithin{equation}{section}

\usepackage{color}

\usepackage[normalem]{ulem}

\def\XXint#1#2#3{{
\setbox0=\hbox{$#1{#2#3}{\int}$}
\vcenter{\hbox{$#2#3$}}\kern-.5\wd0}}

\DeclareMathOperator{\spn}{span}

\begin{document}
\title{Analysis of multivariate Gegenbauer approximation in the hypercube}
\author{Haiyong Wang\footnotemark[1]~\footnotemark[2]~ and Lun Zhang\footnotemark[3]}
\maketitle
\renewcommand{\thefootnote}{\fnsymbol{footnote}}
\footnotetext[1]{School of Mathematics and Statistics, Huazhong
University of Science and Technology, Wuhan 430074, P. R. China.
E-mail: \texttt{haiyongwang@hust.edu.cn}}

\footnotetext[2]{Hubei Key Laboratory of Engineering Modeling and
Scientific Computing, Huazhong University of Science and Technology,
Wuhan 430074, P. R. China.}

\footnotetext[3]{School of Mathematical Sciences, Fudan University,
Shanghai 200433, P. R. China. E-mail: \texttt{lunzhang@fudan.edu.cn}}

\begin{abstract}
In this paper, we are concerned with multivariate Gegenbauer
approximation of functions defined in the $d$-dimensional hypercube. Two new and sharper bounds for the coefficients of
multivariate Gegenbauer expansion of analytic functions are presented based on two
different extensions of the Bernstein ellipse. We then establish an
explicit error bound for the multivariate Gegenbauer approximation
associated with an $\ell^q$ ball index set in the uniform norm. 
We also consider the multivariate approximation of functions with finite regularity and derive the associated error bound on the full grid in the uniform norm.
As an application, we extend our arguments to obtain some new tight bounds for the coefficients of tensorized Legendre expansions in the context of polynomial
approximation of parameterized PDEs.

\end{abstract}


\vspace{0.05in}


\section{Introduction}\label{sec:introduction}
Let $f(\vec{x})$ be a function defined in the $d$-dimensional
hypercube
\begin{equation}\label{def:hypercube}
\Omega_d:=[-1,1]^d, \qquad d\geq 1.
\end{equation}
An efficient and accurate approximation of $f(\vec{x})$ is to expand
it in terms of tensor products of orthogonal polynomials. Besides
many well-known applications of such kind of expansions for the
univariate case (i.e., $d=1$), they have also been widely used in a
variety of practical problems encountered in higher dimensions. For
example, just to name a few, the tensorized Legendre expansion is an
important tool to approximate the solutions of a large class of
parametrized elliptic PDEs with stochastic coefficients
\cite{beck2014convergence,cohen2011analytic,tran2017analysis}. The
bivariate Chebyshev expansion plays an important role in the fast
solution method developed for Fredholm integral equation of the
second kind \cite{reichel1989fast} and the rapid evaluation of the
Bessel functions of real orders and arguments
\cite{bremer2018bessel}, while the bivariate Jacobi expansions have
been used to analyze the convergence of the $h$-$p$ version of the
finite element solution on quasi-uniform meshes
\cite{guo2007optimal}.


When using polynomial approximations, a fundamental issue is to
estimate their convergence rates or establish some error bounds,
which leads to intensive investigations in the literature. For the
univariate case, the Chebyshev expansion was first considered in
\cite{bernstein1912} (see also
\cite{elliott1964chebyshev,liu2017new,trefethen2013atap} for further
studies), and has been considerably extended to other polynomial
expansions since then (cf.
\cite{wang2016gegenbauer,wang2018legendre,wang2012convergence,xiang2012error,zhao2013sharp}
and the references therein). The multivariate case (i.e., $d\geq2$),
however, remains a research topic of great current interest, and
some important progresses have been achieved over the past decades.
Unlike the univariate case, a proper multi-index set has to be fixed
for the multivariate polynomial approximation. Some
popular choices include the hyperbolic cross index set and those
induced from the $1$- and $\infty$- norms of the multi-index. 
An error estimate of the tensorized Legendre expansion on the full
grid (i.e., the index set induced from the $\infty$-norm of the
multi-index) can be found in \cite{canuto2006spectral}, evaluated in
the Sobolev space. Shen and Wang in \cite{shen2010sparse} analyzed
the Jacobi approximations on the full grid and hyperbolic cross
Jacobi approximations in the context of anisotropic Jacobi weighted
Korobov spaces. More recently, based on a new observation, Trefethen
introduced the Euclidean degree for the multivariate polynomial in
\cite{trefethen2017cubature}, and further obtained the convergence
rate of the tensorized Chebyshev expansions for analytic functions with the multi-index sets
induced from $1$-, $2$-, and $\infty$- norms of the multi-index in
\cite{trefethen2017multivatiate}; see also the work of Bos and
Levenberg \cite{bos2018bernstein} for the studies from the viewpoint
of Bernstein-Walsh theory.


In this paper, we first establish some new and explicit bounds for
the coefficients of multivariate Gegenbauer expansion of analytic functions. This can also
be viewed as an extension of the results in
\cite{wang2016gegenbauer} for the univariate case to the
multivariate setting. We then apply these explicit bounds to
derive an explicit error bound for the multivariate Gegenbauer
approximation associated with an $\ell^q$ ball index set, which
particularly include the approximations with the index sets induced
from $1$-, $2$-, and $\infty$- norms of the multi-index as special
cases. For isotropic functions which are rotationally invariant, we
observe numerically that the error estimates obtained agree well
with the empirical rates. We next give a brief discussion on the multivariate approximation of functions with finite regularity and obtain the associated error bound on the full grid in the uniform norm. Finally, as an application, we show that
our arguments can be extended to obtain some new tight bounds on the coefficients
of tensorized Legendre expansions arising from polynomial approximation for a family
of parameterized PDEs.



The rest of this paper is organized as follows. In Section
\ref{sec:property}, we collect some basic properties of Gegenbauer
polynomials and give an explicit bound for the weighted Cauchy
transform of the Gegebauer polynomials for later use. In Sections
\ref{sec:MultiExpan} and \ref{sec:MultiExpRate}, we focus on the multivariate Gegenbauer approximation of analytic functions.
More precisely, two explicit bounds for the
coefficients of multivariate Gegenbauer expansion based on two
different assumptions on $f(\vec{x})$ are derived in Section \ref{sec:MultiExpan}, which allows us to establish an explicit error bound for the
multivariate Gegenbauer approximation associated with an $\ell^q$ ball index set in Section \ref{sec:MultiExpRate}, where the theoretical results are also illustrated
in numerical experiments. In Section \ref{sec:diff}, we consider the multivariate Gegenbauer approximation of a class of functions with finite regularity, and obtain the bound for the coefficients of the expansion as well as the error bound for the approximation on the full grid in the uniform norm.
In Section \ref{sec:appinPDE}, we discuss an application of our results to polynomial approximation of
parameterized PDEs. We finish the paper with some concluding remarks in Section \ref{sec:conclusion}.

\section{Preliminaries}\label{sec:property}

It is the aim of this section to make some preparations for our
later analysis. We first give a brief review of the basic properties
of Gegenbauer polynomials $C_n^{(\lambda)}(x)$, and then present an
explicit optimal upper bound of weighted Cauchy transform of
$C_n^{(\lambda)}(x)$ on the Bernstein ellipse.


\subsection{Gegenbauer polynomials}
\label{sec:gegenbauer}


The Gegenbauer polynomials $C_n^{(\lambda)}(x)$ are polynomials of
degree $n$ orthogonal over the interval $\Omega_1=[-1,1]$ with
respect to the weight function
$\omega_{\lambda}(x) = (1-x^2)^{\lambda-\frac{1}{2}}, \lambda>-\frac{1}{2}$.
More precisely, we have
\begin{equation}\label{eq:orthogonality}
\int_{\Omega_1}  C_{m}^{(\lambda)}(x) C_{n}^{(\lambda)}(x)
\omega_{\lambda}(x) \mathrm{d}x = h_n^{(\lambda)}  \delta_{m,n},
\end{equation}
where $\delta_{m,n}$ is the Kronnecker delta and
\begin{align}\label{eq:normalization constant}
h_n^{(\lambda)} = \frac{2^{1-2\lambda} \pi}{ \Gamma(\lambda)^2 }
\frac{\Gamma(n+2\lambda)}{\Gamma(n+1) (n+\lambda)}, \qquad
\lambda\neq 0,
\end{align}
with $\Gamma(z)$ being the usual gamma functions (cf. \cite[Chapter 5]{DLMF}).
The Gegenbauer polynomials are fixed by requiring
\begin{equation}\label{eq:normalization gegenbauer}
C_{n}^{(\lambda)}(1) =
\frac{\Gamma(n+2\lambda)}{n!\Gamma(2\lambda)}, \qquad
\lambda>-\frac{1}{2},\qquad \lambda\neq 0,
\end{equation}
If $\lambda = 0$, we have $C_0^{(0)}(x) = 1$ and $C_n^{(0)}(1) =
2/n$ for $n \geq 1$.
Furthermore, Gegenbauer polynomials satisfy the following inequality (cf. \cite[Equation 18.14.4]{DLMF})
\begin{align}\label{eq:GegenBound}
\left|C_{n}^{(\lambda)}(x)\right| \leq C_{n}^{(\lambda)}(1), \quad |x|\leq1, \quad
\lambda>0, \quad n\geq 0.
\end{align}

\subsection{An explicit optimal upper bound of weighted Cauchy transform of $C_n^{(\lambda)}(x)$ on the Bernstein ellipse}

For $z\in \mathbb{C}\setminus \Omega_1$, we define
\begin{align}\label{def:Qn}
\mathcal{Q}_{n}^{(\lambda)}(z) := \left\{
                                    \begin{array}{ll}
                                      {\displaystyle\frac{1}{2h_n^{(\lambda)}} \int_{\Omega_1}
\frac{\omega_{\lambda}(x) C_{n}^{(\lambda)}(x)}{z-x} \mathrm{d}x}, &
\hbox{$\lambda \neq 0$,}
\\
[15pt]
                                     {\displaystyle\lim_{\lambda\rightarrow0}
\mathcal{Q}_{n}^{(\lambda)}(z)} , & \hbox{$\lambda=0,~n=0$,}
\\
[15pt]                                    {\displaystyle\frac{2}{n}
\lim_{\lambda\rightarrow0} \lambda \mathcal{Q}_{n}^{(\lambda)}(z)}
, & \hbox{$\lambda=0,~n\geq 1$,}
                                    \end{array}
                                  \right.
\end{align}
where $h_n^{(\lambda)}$ is given in \eqref{eq:normalization constant}. When $\lambda=0$, it is easily seen from \eqref{eq:chebyandGegen} that
\begin{align}\label{eq:QnT}
\mathcal{Q}_n^{(0)}(z) = \frac{1}{2 h_{n}^{(0)}} \int_{\Omega_1}
\frac{\omega_0(x) T_n(x)}{z-x} \mathrm{d}x,
\end{align}
where $h_0^{(0)} = \pi$ and $h_n^{(0)}=\pi/2$ for $n\geq1$. Thus, up to some constant term,
$\mathcal{Q}_{n}^{(\lambda)}(z)$ is the weighted Cauchy transform of
$C_n^{(\lambda)}(x)$ (for $\lambda\neq 0$) or $T_n(x)$ (for $\lambda=0$), which is analytic in the whole complex plane with
a cut along $\Omega_1$.

We need an explicit upper bound of $\mathcal{Q}_{n}^{(\lambda)}$ for $z$ belonging to the so-called Bernstein ellipse, which is crucial in our subsequent analysis.
\begin{definition}
The Bernstein ellipse $\mathcal{E}_{\rho}$ is
defined by
\begin{equation}\label{def:BernsteinEllip}
\mathcal{E}_{\rho} =  \left\{ z \in \mathbb{C} ~\bigg|~ z =
\frac{1}{2} \big( u + u^{-1} \big),~~ |u| = \rho >1 \right\},
\end{equation}
which has the foci at $\pm 1$ with the major and minor semi-axes given by
$\frac{1}{2}(\rho+\rho^{-1})$ and $\frac{1}{2}(\rho-\rho^{-1})$,
respectively.
\end{definition}

By combining Corollary 3.4 and Theorems 4.1 and 4.3 in \cite{wang2016gegenbauer}, we then have the following
explicit optimal upper bound of $\mathcal{Q}_{n}^{(\lambda)}(z)$ over the Bernstein ellipse $\mathcal{E}_{\rho}$.
\begin{proposition}\label{prop:upperbound}
For $z\in \mathcal{E}_{\rho}$ and $\lambda>0$, we have
\begin{equation}\label{eq:upperbound}
\left| \mathcal{Q}_n^{(\lambda)}(z) \right| \leq\left\{
                                                  \begin{array}{ll}
                                                    \overline{D}_{\rho}^{(\lambda)}, & \hbox{$n=0$,}
\\[2ex]
D_{\rho}^{(\lambda)} \frac{n^{1-\lambda}}{\rho^{n}}, &
\hbox{$n \geq 1$,}
                                                  \end{array}
                                                \right.
\end{equation}
where the $n$-independent constants $\overline{D}_{\rho}^{(\lambda)}$ and $D_{\rho}^{(\lambda)}$ are defined by
\begin{equation}\label{eq:upperboundQ}
\overline{D}_{\rho}^{(\lambda)} = \frac{1}{\rho} \times \left\{
                                                        \begin{array}{ll}
                                                          (1+\rho^{-2})^{\lambda-1}, & \hbox{$\lambda\geq1$,} \\
[2ex]
                                                          (1-\rho^{-2})^{\lambda-1}, & \hbox{$0<\lambda<1$,}
                                                        \end{array}
                                                      \right.
\end{equation}
and
\begin{align}\label{eq:D}
D_{\rho}^{(\lambda)} = \frac{\Gamma(\lambda)}{\rho} \times
\left\{\begin{array}{ll}
                                           \exp\left(\frac{1}{12} \right) (1+\rho^{-2})^{\lambda-1}  , & \mbox{ $\lambda\geq1$}, \\
[2ex]
                                           \exp\left(\frac{1}{12} + \frac{1-\lambda}{2\lambda}\right) (1-\rho^{-2})^{\lambda-1}  , & \mbox{ $0<\lambda<1$}.
                                        \end{array}
                                        \right.
\end{align}
The bound in \eqref{eq:upperbound}, apart from a constant factor, is
optimal as $n\rightarrow\infty$ in the sense that it can not be
improved in any lower power of $n$ further.
\end{proposition}

\begin{remark}
If $\lambda=0$, i.e., for the Chebyshev polynomials of the first kind, we have the following explicit formula for
$\mathcal{Q}_n^{(0)}(z)$:
\begin{align}\label{eq:CauchyChebT}
\mathcal{Q}_n^{(0)}(z) = \left\{\begin{array}{ll}
                                          {\displaystyle \frac{1}{\sqrt{z^2-1} (z\pm\sqrt{z^2-1})^{n}}  } , & \mbox{$n\geq1$}, \\
[2ex]
                                          {\displaystyle \frac{1}{2\sqrt{z^2-1}}  }, & \mbox{$n=0$.}
                                        \end{array}
                                        \right.
\end{align}
If $\lambda=1$, i.e., for the Chebyshev polynomials of the second
kind,
we have
\begin{align}\label{eq:CauchyChebU}
\mathcal{Q}^{(1)}_n(z) = \frac{1}{( z \pm \sqrt{z^2 - 1} )^{n+1}},
\quad n\geq0.
\end{align}
This particularly implies that
\begin{equation}\label{eq:boundQU}
\left| \mathcal{Q}^{(1)}_n(z) \right| \leq \frac{1}{\rho^{n+1}}, \qquad z\in \mathcal{E}_{\rho},
\end{equation}
i.e., the prefactor $D_\rho^{(1)}$ in \eqref{eq:upperbound} can be improved to be $1/\rho$.
\end{remark}

\section{Multivariate Gegenbauer expansion of analytic functions}\label{sec:MultiExpan}

In this section, we intend to estimate the coefficients of the multivariate Gegenbauer expansion of analytic functions based on two different assumptions on the analyticity.

\subsection{Notations}
We first introduce some notations to be used throughout the rest of this paper.

\begin{itemize}
\item
We shall denote by $\vec{x}$ and $\vec{z}$
the point in $\mathbb{R}^d$ and $\mathbb{C}^d$, respectively, i.e.,
\begin{align}
\mathbf{x} = (x_1,\ldots,x_d), \qquad  \mathbf{z} =
(z_1,\ldots,z_d).
\end{align}

\item The notation $\mathbb{N}_0^d$ stands for the set of all $d$-tuples
$\vec{k}=(k_1,k_2,\ldots,k_d)$, where
$k_i\in\mathbb{N}_0=\{0,1,2,\ldots\}.$ Such a $d$-tuple is called a
multi-index. For any two multi-indices $\vec{k}=(k_1,\ldots,k_d)$
and $\vec{t}=(t_1,\ldots,t_d)$, we define the following
componentwise operation
$
\vec{k}+\vec{t}=(k_1+t_1,\ldots,k_d+t_d),
$
and use the convention
$
\vec{k}\leq \vec{t} \Leftrightarrow k_j \leq t_j, j=1,2,\ldots d
$.

\item
Let $\mathbf{1}=(1,\ldots,1)\in \mathbb{N}_0^d$. For a scalar $t\in \mathbb{R}$, we define
$
\mathbf{k} + t = \mathbf{k} + t\cdot \mathbf{1}=(k_1+t,\ldots,k_d+t)
$
and
$
\vec{k}^t = \prod_{j=1}^{d} k_j^t
$.

\item
If $\phi$, $\phi_{k_j}$, $j=1,\ldots,d$, are functions of one
variable, we define
\begin{align}
\phi(\mathbf{x}) = \prod_{j=1}^{d} \phi(x_j), \qquad
\phi_{\mathbf{k}}(\mathbf{x}) = \prod_{j=1}^{d} \phi_{k_j}(x_j).
\end{align}
Thus, $\vec{x}^{\vec{k}}=\prod_{j=1}^d x_j^{k_j}$ is a multivariate
monomial.

\item We define
\begin{align}\label{eq:Norm} \|\vec{k}\|_{q} :=
\left\{\begin{array}{ll}
               {\displaystyle (k_1^q+\cdots+k_d^q)^{\frac{1}{q}} },   & \mbox{$0<q<\infty$, }\\ [2ex]
               {\displaystyle \max_{1\leq i\leq d} k_i },    & \mbox{$q=\infty$}.
                  \end{array}
                  \right.
\end{align}

\item Given a multi-index $\vec{k}=(k_1,\ldots,k_d)$ and a multivariate function $f(\vec{x})$, we denote the $||\vec{k}||_1$th mixed partial derivative by
$|\vec{k}|=k_1+\cdots+k_d$ and
\begin{equation}\label{def:derivative}
\partial^{\vec{k}} f = \frac{\partial^{||\vec{k}||_1} f}{\partial x_1^{k_1}\cdots \partial
x_d^{k_d}}
=\partial x_1^{k_1}\cdots \partial
x_d^{k_d}f.
\end{equation}

\item For any two multi-indices $\vec{m}=(m_1,\ldots,m_d)$ and $\vec{n}=(n_1,\ldots,n_d)$, we set
\begin{equation}\label{def:minindices}
\min\{\vec{m},\vec{n}\}:=(\min\{m_1,n_1\},\ldots,\min\{m_d,n_d\}).
\end{equation}
\end{itemize}


\subsection{Multivariate Gegenbauer expansion}

Let $f(\mathbf{x})$ be an analytic function defined in the hypercube
$\Omega_d$. The multivariate Gegenbauer series expansion of $f(\mathbf{x})$ is defined by
\begin{align}\label{eq:gegenbauerexpansion}
f(\mathbf{x}) = \sum_{\mathbf{k}\in \mathbb{N}_0^d} a_\vec{k}
C_{\vec{k}}^{(\lambda)}(\mathbf{x}),
\end{align}
where $C_{\vec{k}}^{(\lambda)}(\mathbf{x})=\prod_{i=1}^d C_{k_i}^{(\lambda)}(x_i)$ stands for the tensorized Gegenbauer polynomials, and by orthogonality \eqref{eq:orthogonality},
\begin{equation}\label{def:akintegral}
a_{\vec{k}} = \frac{1}{h_{\vec{k}}^{(\lambda)}} \int_{\Omega_d}
f(\mathbf{x}) C_{\mathbf{k}}^{(\lambda)}(\mathbf{x})
\omega_{\lambda}(\mathbf{x}) \mathrm{d}\mathbf{x}
\end{equation}
with $\mathrm{d}\vec{x}=\prod_{i=1}^d \mathrm{d}x_i$ and
$h_{\vec{k}}^{(\lambda)}=\prod_{i=1}^d h_{k_i}^{(\lambda)}$. We refer to \cite{Mason} and references therein for the convergence issue of multivariate Gegenbauer series expansions.

We are interested in the estimate of the expansion coefficients
$a_{\vec{k}}$. The case of a single variable, i.e., $d=1$, is well
established; cf. \cite{wang2016gegenbauer} and references therein.
To deal with the higher dimensional case $d>1$, an essential issue
here is to extend the Bernstein ellipse to a region in
$\mathbb{C}^d$. In what follows, we divide our discussions on the
estimate of $a_{\vec{k}}$ into two cases, based on different
extensions of the Bernstein ellipse.

%
%
%

\subsection{Estimates of $a_{\vec{k}}$ under Assumption \ref{assump1} on $f$}
\label{sec:assumpt1}

A natural extension of the Bernstein ellipse $\mathcal{E}_{\rho}$ to $\mathbb{C}^d$ is the polyellipse, and we then make the following assumption on $f$.

\begin{assum}\label{assump1}
The function $f$ is analytic inside the polyellipse
\begin{equation}
\mathcal{E}_{\pmb{\rho}} : = \bigotimes_{j=1}^{d}
\mathcal{E}_{\rho_j},
\end{equation}
where $\mathcal{E}_{\rho}$ is defined in \eqref{def:BernsteinEllip},
and $\pmb{\rho}=(\rho_1,\ldots,\rho_d)$ with $\rho_i>1$, $i=1,\ldots, d$.
\end{assum}

The main result of this section is the following theorem.

\begin{theorem}\label{thm:akbound1}
Under Assumption \ref{assump1} and for $\lambda>0$, the multivariate
Gegenbauer coefficients of $f(\mathbf{x})$ satisfy
\begin{align}\label{eq:akupbound1}
\left|a_{\vec{k}}\right| & \leq
\frac{\mathcal{B}_{f}L(\mathcal{E}_{\pmb{\rho}})}{\pi^d
\pmb{\rho}^{\vec{k}}} \prod_{\substack{1\leq i \leq d \\ k_i =
0}}\overline{D}_{\rho_i}^{(\lambda)} \prod_{\substack{1\leq j \leq d
\\ k_j \neq
0}} k_j^{1-\lambda} D_{\rho_j}^{(\lambda)} ,
\end{align}
where
\begin{align}\label{eq:MaxValue}
\mathcal{B}_{f} = \max_{\mathbf{z}\in
\mathcal{E}_{\pmb{\rho}}}|f(\mathbf{z})|,
\end{align}
$L(\mathcal{E}_{\pmb{\rho}}):=\prod_{i=1}^d L(\mathcal{E}_{\rho_i})$
with $L(\mathcal{E}_{\rho_i})$ being the length of the circumference
of the Bernstein ellipse $\mathcal{E}_{\rho_i}$, and the constants
$\overline{D}_{\rho_i}^{(\lambda)}$, $D_{\rho_j}^{(\lambda)}$ are
defined in \eqref{eq:upperboundQ} and \eqref{eq:D}, respectively. In
addition, apart from some constant factor, the bound
\eqref{eq:akupbound1} is optimal as $k_j \to  +\infty$ for
$j=1,\ldots,d$.
\end{theorem}

\begin{proof}
Since $f(\mathbf{z})$ is analytic inside the Bernstein polyellipse
$\mathcal{E}_{\pmb{\rho}}$, thus, analytic in $\Omega_d$. By
Cauchy's integral formula for the analytic function of several
variables (cf. \cite[Page 32]{bochner1948complex}), we have
\begin{align}\label{eq:Cauchyformula}
f(\mathbf{x}) = \left( \frac{1}{2\pi i} \right)^d
\oint_{\mathcal{E}_{\pmb{\rho}}}
\frac{f(\mathbf{z})}{\mathbf{z}-\mathbf{x}} \mathrm{d}\mathbf{z},
\end{align}
where $\mathbf{z}-\mathbf{x} = \prod_{j=1}^{d} (z_j - x_j)$. Inserting \eqref{eq:Cauchyformula} into \eqref{def:akintegral}, it then follows from interchanging the order of
integration that
\begin{align}\label{eq:coutour}
a_{\vec{k}}
= \left( \frac{1}{\pi i} \right)^d \oint_{\mathcal{E}_{\pmb{\rho}}}
f(\mathbf{z}) \mathcal{Q}_{\mathbf{k}}^{(\lambda)}(\mathbf{z})
\mathrm{d}\mathbf{z},
\end{align}
where recall that
$\mathcal{Q}_{\mathbf{k}}^{(\lambda)}(\mathbf{z})=\prod_{i=1}^d
\mathcal{Q}_{k_i}^{(\lambda)}(z_i)$ with
$\mathcal{Q}_{k_i}^{(\lambda)}(z)$ defined in \eqref{def:Qn}.

As a consequence, it is readily seen that
\begin{align}\label{eq:absak}
|a_{\vec{k}}| &\leq
\frac{\mathcal{B}_{f}L(\mathcal{E}_{\pmb{\rho}})}{\pi^d}
\max_{\mathbf{z}\in \mathcal{E}_{\pmb{\rho}}} \left|
Q_{\mathbf{k}}^{(\lambda)}(\mathbf{z})\right| =
\frac{\mathcal{B}_{f}L(\mathcal{E}_{\pmb{\rho}})}{\pi^d}
\prod_{i=1}^d \max_{z_i \in \mathcal{E}_{\rho_i}} \left|
Q_{k_i}^{(\lambda)}(z_i)\right|.
\end{align}
The upper bound of $a_{\vec{k}}$ in \eqref{eq:akupbound1} and its
optimality follows directly by combining \eqref{eq:absak} with
Proposition \ref{prop:upperbound}. This completes the proof of
Theorem \ref{thm:akbound1}.
\end{proof}

As mentioned at the end of Section \ref{sec:gegenbauer}, the classical Chebyshev polynomials and Legendre polynomials are special cases of Gegenbauer polynomials. Since these classical polynomials play important roles in practice, we next state the relevant results for these polynomials.

\begin{corollary}\label{coro:MulChebT}
Suppose that the multivariate function $f$ satisfies Assumption
\ref{assump1} and consider the following tensorized Chebyshev
expansion of the first kind:
\begin{align}\label{eq:MulChebT}
f(\vec{x}) = \sum_{\mathbf{k}\in \mathbb{N}_0^d} a_{\mathbf{k}}^{T}
T_{\vec{k}}(\vec{x}), \quad a_{\vec{k}}^{T} =
\frac{2^{d-\aleph(\mathbf{k})}}{\pi^d}  \int_{\Omega_d} f(\vec{x})
 T_{\vec{k}}(\vec{x}) \omega_{0}(\mathbf{x}) \mathrm{d}\mathbf{x},
\end{align}
with $\aleph(\mathbf{k}):=\#\{i:k_i=0\}$. Then, we have
\begin{align}\label{eq:chebybound}
\left|a_{\mathbf{k}}^{T}\right| \leq 2^{d-\aleph(\mathbf{k})} \frac{
\mathcal{B}_{f}}{\pmb{\rho}^{\vec{k}}}.
\end{align}
\end{corollary}
\begin{proof}
The proof is similar to that of Theorem \ref{thm:akbound1}. By
Cauchy's integral formula and \eqref{eq:QnT}, it is easily seen that
\begin{align}\label{eq:chebycontour}
a_{\mathbf{k}}^{T} = \left( \frac{1}{\pi i} \right)^d
\oint_{\mathcal{E}_{\pmb{\rho}}} f(\mathbf{z})
\mathcal{Q}_{\mathbf{k}}^{(0)}(\mathbf{z}) \mathrm{d}\mathbf{z}.
\end{align}
We now make change of variables $z_j = (u_j + u_j^{-1} )/2$ with $u_j \in \mathcal{C}_{\rho_j}:=\{z\in\mathbb{C} ~|~ |z|=\rho_j\}$
for each $j=1,\ldots,d$ in \eqref{eq:CauchyChebT}. A simple calculation shows that
\begin{align}
\mathcal{Q}_{k_j}^{(0)}(z_j) = \left\{\begin{array}{ll}
                                          {\displaystyle  \frac{2}{u_j^{k_j}(u_j - u_j^{-1})}  } , & \mbox{$k_j\geq1$}, \\
[15pt]
                                          {\displaystyle \frac{1}{u_j-u_j^{-1}}  }, & \mbox{$k_j=0$.}
                                        \end{array}
                                        \right.
\end{align}
Consequently,
\begin{align}\label{eq:chebycontour2}
a_{\mathbf{k}}^{T}
&= \frac{1}{2^{\aleph(\mathbf{k})}} \left( \frac{1}{\pi i} \right)^d
\oint_{\mathcal{C}_{\pmb{\rho}}} f(\mathbf{z}(\mathbf{u}))
\prod_{\substack{1\leq i \leq d \\ k_i = 0}} \frac{1}{u_i}
\prod_{\substack{1\leq j \leq d
\\ k_j \neq 0}} \frac{1}{u_j^{k_j+1}} \mathrm{d}\mathbf{u}
\nonumber \\
&= \frac{1}{2^{\aleph(\mathbf{k})}}  \left( \frac{1}{\pi i}
\right)^d \oint_{\mathcal{C}_{\pmb{\rho}}}
\frac{f(\mathbf{z}(\mathbf{u}))}{\mathbf{u}^{\mathbf{k}+1 }}
\mathrm{d}\mathbf{u},
\end{align}
where $\mathcal{C}_{\pmb{\rho}} :=  \bigotimes_{j=1}^{d}
\mathcal{C}_{\rho_j}$ is the polycircle.

The desired result \eqref{eq:chebybound} follows directly from the above formula.
\end{proof}

\begin{remark}
If $d=1$,  the bound \eqref{eq:chebybound} reduces to
\[
|a_k^{T}| \leq \left\{\begin{array}{ll}
                              {\displaystyle \mathcal{B}_{f} } , & \mbox{$k=0$},
                              \\[2ex]
                              {\displaystyle \frac{2\mathcal{B}_{f}}{\rho^k} }, & \mbox{$k\geq1$.}
                                        \end{array}
                                        \right.
\]
Thus, we have recovered the sharpest bound which was first obtained
by Bernstein in \cite{bernstein1912}. For $d\geq 2$, the bound
\eqref{eq:chebybound} can also be found in
\cite[Page 95]{bochner1948complex}, up to the explicit prefactor.
\end{remark}


On account of \eqref{eq:boundQU} and \eqref{eq:absak}, the following corollary concerning
tensorized Chebyshev expansion of the second kind is immediate.
\begin{corollary}\label{coro:MulChebU}
Suppose that the multivariate function $f$ satisfies Assumption
\ref{assump1} and consider the following tensorized Chebyshev
expansion of the second kind
\begin{align}
f(\vec{x}) = \sum_{\mathbf{k}\in \mathbb{N}_0^d} a_{\mathbf{k}}^{U}
U_{\vec{k}}(\vec{x}), \quad a_{\mathbf{k}}^{U} =
\frac{1}{h_{\vec{k}}^{(1)}} \int_{\Omega_d} f(\mathbf{x})
U_{\mathbf{k}}(\mathbf{x}) \omega_{1}(\mathbf{x})
\mathrm{d}\mathbf{x}.
\end{align}
Then, we have
\begin{align}
\left|a_{\mathbf{k}}^{U}\right| \leq
\frac{\mathcal{B}_{f}L(\mathcal{E}_{\pmb{\rho}})}{\pi^d\pmb{\rho}^{\vec{k}+1}}.
\end{align}
\end{corollary}

Finally, the tensorized Legendre expansion is defined by
\begin{align}\label{def:tensorLeg}
f(\vec{x}) = \sum_{\mathbf{k}\in \mathbb{N}_0^d} a_{\mathbf{k}}^{L}
P_{\vec{k}}(\vec{x}), \quad  a_{\mathbf{k}}^{L} =
\frac{1}{h_{\vec{k}}^{(\frac{1}{2})}} \int_{\Omega_d}  f(\mathbf{x})
P_{\mathbf{k}}(\mathbf{x}) \mathrm{d}\mathbf{x}.
\end{align}
where $P_{\vec{k}}(\vec{x}) = \prod_{i=1}^{d} P_{k_i}(x_i)$, with
$P_k(x)$ defined as in \eqref{eq:gegenbauer and legendre}. Let
$\overline{P}_k(x)$ be the normalized Legendre polynomial of degree
$k$, i.e., $\overline{P}_k(x) = \sqrt{\frac{2k+1}{2}} P_k(x)$. The
normalized Legendre expansion is defined by
\begin{align}\label{eq:normLegExp}
f(\vec{x}) = \sum_{\mathbf{k}\in \mathbb{N}_0^d}
\overline{a}_{\mathbf{k}}^{L} \overline{P}_{\vec{k}}(\vec{x}), \quad
\overline{a}_{\mathbf{k}}^{L} = \int_{\Omega_d} f(\mathbf{x})
\overline{P}_{\mathbf{k}}(\mathbf{x}) \mathrm{d}\mathbf{x},
\end{align}
where $\overline{P}_{\vec{k}}(\vec{x}) = \prod_{i=1}^{d}
\overline{P}_{k_i}(x_i)$. Both kinds of Legendre expansion are
frequently used in practice. By setting $\lambda=1/2$ in
\eqref{eq:akupbound1} and note that $\overline{a}_{\mathbf{k}}^{L} =
a_{\mathbf{k}}^{L}/\sqrt{\vec{k}+\frac{1}{2}}$, we finally obtain
the estimates of $a_{\mathbf{k}}^{L}$ and
$\overline{a}_{\mathbf{k}}^{L}$ in the following corollary.
\begin{corollary}
Under Assumption \ref{assump1}, we have 
\begin{align}\label{eq:LegCoeff}
\left|a_{\mathbf{k}}^{L} \right| &\leq \frac{
\mathcal{B}_{f}L(\mathcal{E}_{\pmb{\rho}})}{\pi^d
\pmb{\rho}^{\vec{k}}} \prod_{\substack{1\leq i \leq d
\\ k_i = 0}}\overline{D}_{\rho_i}^{(1/2)}
\prod_{\substack{1\leq j \leq d \\ k_j \neq 0}} \sqrt{k_j}
D_{\rho_j}^{(1/2)},
\end{align}
and
\begin{align}\label{eq:NormLegCoeff}
\left|\overline{a}_{\mathbf{k}}^{L} \right| &\leq
\frac{2^{\frac{\aleph(\mathbf{k})}{2}}
\mathcal{B}_{f}L(\mathcal{E}_{\pmb{\rho}})}{\pi^d
\pmb{\rho}^{\vec{k}}} \prod_{\substack{1\leq i \leq d
\\ k_i = 0}}\overline{D}_{\rho_i}^{(1/2)}
\prod_{\substack{1\leq j \leq d \\ k_j \neq 0}} D_{\rho_j}^{(1/2)},
\end{align}
where the constants $\overline{D}_{\rho_i}^{(\frac{1}{2})}$,
$D_{\rho_j}^{(\frac{1}{2})}$ are defined in \eqref{eq:upperboundQ}
and \eqref{eq:D}, respectively.
\end{corollary}

\subsection{Numerical experiments and Assumption II on $f$}
Although we have derived an explicit bound for the coefficients of
multivariate Gegenbauer expansion under Assumption I on $f$, it is
unclear how to determine an optimal polyellipse such that the bound
matches the decay rate of the coefficients well. To introduce our
second assumption, we proceed to perform numerical experiments to
the multivariate normalized Legendre coefficients
$\overline{a}_{\vec{k}}^L$ for the following two bivariate functions
\begin{equation}\label{def:f1}
f_1(x_1,x_2) = \sqrt{x_1^2+x_2^2+\frac{1}{2}},
\end{equation}
and
\begin{equation}\label{def:f2}
f_2(x_1,x_2) = \frac{1}{x_1^2+x_2^2+1}.
\end{equation}
Note that both functions are isotropic and are analytic for all real
values of $x_1$ and $x_2$. Moreover, for complex values of $x_1$ and
$x_2$, the former function has a branch point at $x_1^2+x_2^2=-1/2$
and the latter function has a pole at $x_1^2+x_2^2=-1$. Contour
plots of $\left|\overline{a}_{\vec{k}}^L\right|$ are shown in Figure
\ref{fig:contourplot}. In both cases, we observe clearly that the
contours look like circular arcs in the positive orthant. This
phenomena was first reported in \cite{trefethen2017cubature} for the
multivariate Chebyshev coefficients of isotropic functions.


\begin{figure}[ht]
\centering
\includegraphics[width=7.2cm,height=6.8cm]{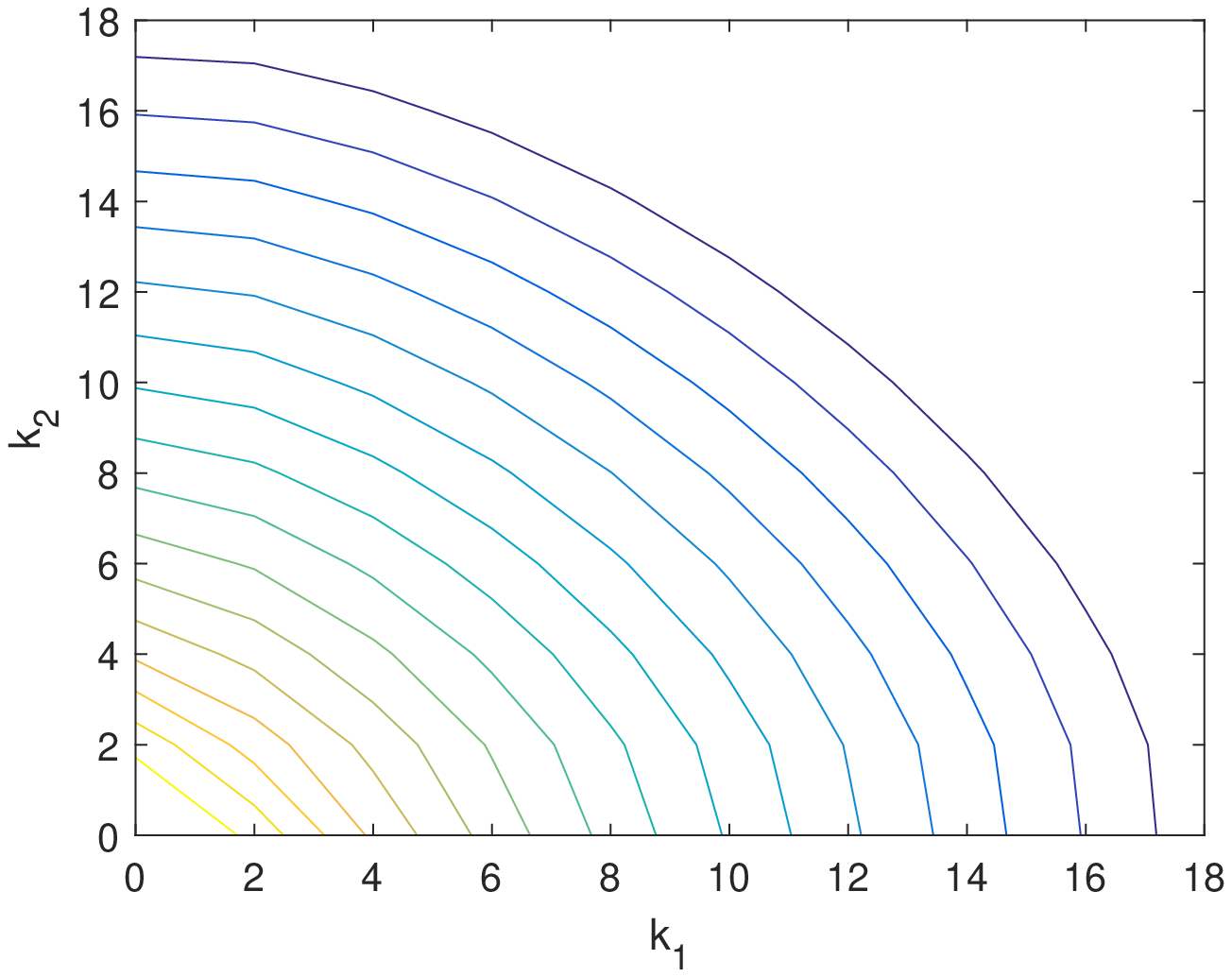}~~
\includegraphics[width=7.2cm,height=6.8cm]{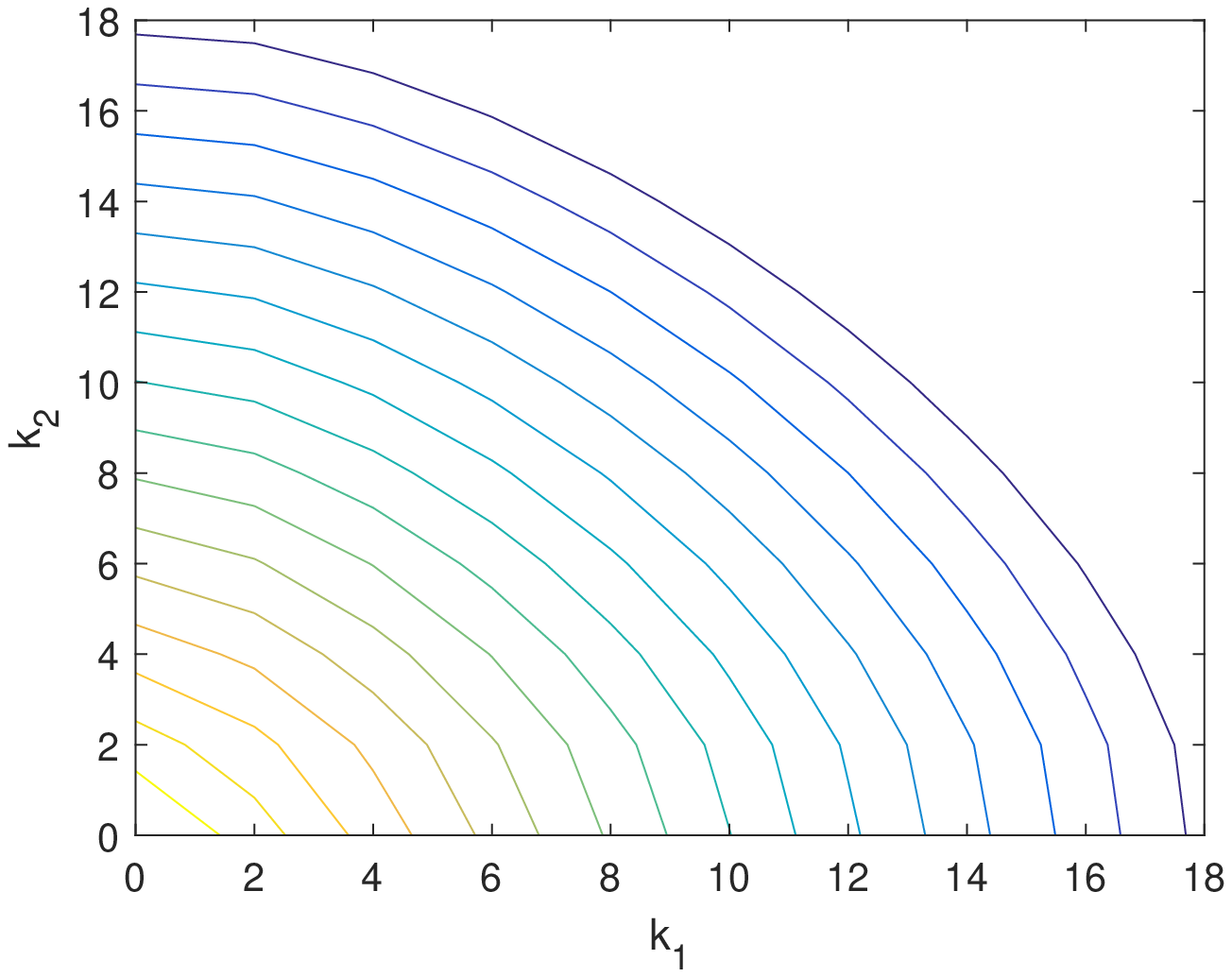}
\caption{Contour plots of
$\left|\overline{a}_{\mathbf{k}}^{L}\right|$ for the bivariate
functions \eqref{def:f1} (left) and \eqref{def:f2} (right). From
inside out, the contours represent
$10^{-1},10^{-2},\ldots,10^{-16}$. } \label{fig:contourplot}
\end{figure}

To approximate a multivariate function $f$ in $\Omega_d$ by a multivariate polynomial, it is usual to use the so-called total degree $d_{T}$ or maximal degree $d_{M}$ of the multivariate polynomial. More precisely, for a multivariate monomial $\vec{x}^{\vec{k}}$, we set
\begin{equation}\label{def:dTM}
d_{T}(\vec{x}^{\vec{k}}):=\|\vec{k}\|_1,\qquad
d_{M}(\vec{x}^{\vec{k}}):=\|\vec{k}\|_{\infty},
\end{equation}
and the degree of a multivariate polynomial is then defined as the
maximum of the degrees of its nonzero monomials. The above
observation, however, implies that any approximations based on these
traditional notions might be suboptimal. This invokes Trefethen in
\cite{trefethen2017cubature} to introduce the following
Euclidean degree for $\vec{x}^{\vec{k}}$:
\begin{equation}
d_{E}(\vec{x}^{\vec{k}}):=\|\vec{k}\|_{2}, 
\end{equation}
which also leads to the definition of Euclidean degree of a
multivariate polynomial. Note that the Euclidean degree might not be
an integer. The motivation behind this definition is the
multivariate polynomials with prescribed Euclidean degree may
provide approximations with uniform resolution in all directions for
functions defined in the hypercube $\Omega_d$, as evidenced in
Figure \ref{fig:contourplot}.


As an application of the Euclidean degree, it is used to establish
the rate of decay of the multivariate Chebyshev coefficients in
\cite{trefethen2017multivatiate} by imposing some conditions on
$f$. To some extent, this explains the aforementioned effect in a
mathematical way. In particular, the following region is introduced
therein to extend the Bernstein ellipse.

\begin{definition}\label{def:Nsa}
For any $s,a>0$, we denote by $N_{s,a}\subseteq\mathbb{C}$ the open
region bounded by the ellipse with foci $0$ and $s$, and leftmost
point $-a$.
\end{definition}

Note that
\begin{equation}\label{eq:squareellipse}
z\in \mathcal{E}_{\rho}~~ \Leftrightarrow~~ z^2 \in \partial
N_{1,h^2},
\end{equation}
where $\partial U$ denotes the boundary of a region $U$ and
\begin{align}\label{eq:rhoh}
h=\frac{\rho-\rho^{-1}}{2}, \qquad \rho=h+\sqrt{1+h^2}.
\end{align}
It is then required in \cite{trefethen2017multivatiate} that $f$ is analytic in the $d$-dimensional region defined by
$\sum_{i=1}^d x_i^2 \in N_{d,h^2}$ for some $h>0 $, which clearly extends the analyticity of $f$ in the Bernstein ellipse to a higher dimensional space.


To deal with the case of multivariate Gegenbauer expansion, we will adopt the following assumption,
which is a slight generalization of the one just mentioned.

\begin{assum}\label{assump2}
There exists some $h>0$ such that $f(\vec{z})$ is analytic in the
$d$-dimensional  region $D_{h,\epsilon}$ defined by
\begin{align}\label{eq:LargeRegion}
D_{h,\epsilon}:=\left\{\vec{z}\in\mathbb{C}^d  ~ \bigg |  ~ \sum_{i=1}^d z_i^2 \in N_{d,h^2+d\epsilon} \right \},
\end{align}
where the region $N_{d,h^2+d\epsilon}$ is defined in Definition
\ref{def:Nsa}, $\epsilon>0$ is an arbitrarily small fixed constant
when $0<\lambda<1$, and $\epsilon=0$ when $\lambda\geq1$ or
$\lambda=0$.
\end{assum}

As we shall see later, the region $D_{h,\epsilon}$ actually contains
some polyellipses. The reason why we need $\epsilon >0$ for
$0<\lambda <1$ will be explained in Remark \ref{rk:failure} below.
We next show the upper bound of multivariate Gegenbauer coefficients
under Assumption \ref{assump2}, which extends the results for the
Chebyshev case.

\subsection{Estimates of $a_{\vec{k}}$ under Assumption \ref{assump2}}\label{sec:estAssII}
The main result of this section is the following theorem.

\begin{theorem}\label{thm:akbound2}
Under Assumption \ref{assump2} and $\lambda>0$, the multivariate
Gegenbauer coefficients of $f(\mathbf{x})$ satisfy
\begin{align}\label{eq:akupbound2}
\left|a_{\vec{k}}\right| &\leq
\frac{\widehat{\mathcal{B}}_{f}L(\mathcal{E}_{\pmb{\widehat{\rho}}})}{\pi^d
\rho^{\|\vec{k}\|_2}} \prod_{\substack{1\leq i \leq d \\ k_i =
0}}\overline{D}_{\widehat{\rho}_i}^{(\lambda)}
\prod_{\substack{1\leq j \leq d
\\ k_j \neq
0}} k_j^{1-\lambda} D_{\widehat{\rho}_j}^{(\lambda)},
\end{align}
where $\rho=h+\sqrt{1+h^2}$,
\begin{align}\label{eq:rhoHat}
\widehat{\rho}_j = \sqrt{(c_jh)^2+\epsilon} +
\sqrt{1+(c_jh)^2+\epsilon}, \qquad j=1,\ldots,d,
\end{align}
with $c_j = k_j/\|\vec{k}\|_2$, and the constants $\overline{D}_{\widehat{\rho}_i}^{(\lambda)}$, $D_{\widehat{\rho}_j}^{(\lambda)}$ are defined in \eqref{eq:upperboundQ} and \eqref{eq:D}, respectively. Moreover, $\mathcal{E}_{\widehat{\pmb{\rho}}}  := \bigotimes_{j=1}^{d}
\mathcal{E}_{\widehat{\rho}_j}$ and the constant $\widehat{\mathcal{B}}_{f}$
is defined by
\begin{align}\label{eq:MaxVal2}
\widehat{\mathcal{B}}_{f} = \max_{\mathbf{z}\in
\mathcal{E}_{\vec{\widehat{\pmb{\rho}}}}}|f(\mathbf{z})|.
\end{align}
\end{theorem}
\begin{proof}
We follow the idea in \cite{trefethen2017multivatiate}, which deals
with the multivariate Chebyshev coefficients. For each $\vec{k}\in
\mathrm{N}_0^d$, we define $h_j=c_jh$ with $c_j =
k_j/\|\vec{k}\|_2$, $j=1,\ldots,d$. It is then easily seen that
$h_1^2+\cdots+d_d^2=h^2$. From \cite[Lemma
5.2]{trefethen2017multivatiate}, we have
\begin{align}
N_{1,h_1^2+\epsilon}\oplus \cdots\oplus N_{1,h_d^2+\epsilon}
\subseteq N_{d,\sum_{i=1}^{d} h_i^2+d\epsilon} =
N_{d,h^2+d\epsilon}, \nonumber
\end{align}
where $\oplus$ denotes the Minkowski sum of sets. This, together
with Assumption \ref{assump2} on $f$, implies that $f$ is analytic
in the region defined by $\{\vec{z}\in\mathbb{C}^d~|~ z_j^2 \in
N_{1,h_j^2+\epsilon}, j=1,\ldots,d\}$. On account of
\eqref{eq:squareellipse} and \eqref{eq:rhoh}, we further conclude
that $f$ is analytic in the polyellipse
$\mathcal{E}_{\pmb{\widehat{\rho}}} = \bigotimes_{j=1}^{d}
\mathcal{E}_{{\widehat{\rho}}_j}$ where each $\widehat{\rho}_j$ is
defined in \eqref{eq:rhoHat}. Hence, by Theorem \ref{thm:akbound1},
it follows that
\begin{align}\label{eq:akupbound3}
\left|a_{\vec{k}}\right| & \leq
\frac{\widehat{\mathcal{B}}_{f}L(\mathcal{E}_{\pmb{\widehat{\rho}}})}{\pi^d
\pmb{\widehat{\rho}}^{\vec{k}}} \prod_{\substack{1\leq i \leq d \\
k_i = 0}}\overline{D}_{\widehat{\rho}_i}^{(\lambda)}
\prod_{\substack{1\leq j \leq d
\\ k_j \neq
0}} k_j^{1-\lambda} D_{\widehat{\rho}_j}^{(\lambda)}.
\end{align}
To this end, we see from \cite[Lemma 5.3]{trefethen2017multivatiate} that
\begin{align}\label{eq:LowBound1}
\widehat{\rho}_j \geq c_jh+\sqrt{1+(c_jh)^2} \geq
(h+\sqrt{1+h^2})^{c_j}=\rho^{\frac{k_j}{\|\vec{k}\|_2}},
\end{align}
which implies
\begin{align}\label{eq:LowBound2}
\pmb{\widehat{\rho}}^\vec{k}=\prod_{j=1}^d\widehat{\rho}_j^{k_j}\geq
\rho^{\frac{k_1^2+\cdots+k_d^2}{||\vec{k}||_2}}=\rho^{\|\vec{k}\|_2}.
\end{align}
Combining  \eqref{eq:LowBound2} and \eqref{eq:akupbound3} then gives
us the the bound of $|a_{\vec{k}}|$ given in \eqref{eq:akupbound2}.
This completes the proof of Theorem \ref{thm:akbound2}.
\end{proof}

\begin{remark}\label{rk:failure}
When $0<\lambda<1$, we note that the $\widehat{\rho}_j$-dependent constants
$D_{\widehat{\rho}_j}^{(\lambda)}$ and $\overline{D}_{\widehat{\rho}_j}^{(\lambda)}$ would be infinity as
$\widehat{\rho}_j \to 1$; see \eqref{eq:upperboundQ} and \eqref{eq:D}. By \eqref{eq:rhoHat}, this is indeed the case if $\epsilon=0$ and $c_j=0$ for some $j$. This explains why we have assumed that $\epsilon>0$ when $0<\lambda<1$, so that $\widehat{\rho}_j>1$ for all $j=1,\ldots,d$.
\end{remark}

Since the cases $\lambda=0$ and $\lambda=1$ are of particular interest, we conclude this section with the relevant results
in the following corollary.
\begin{corollary}\label{cor:multcheby}
Let $f$ be analytic in the $d$-dimensional region $D_{h,0}$ defined
in \eqref{eq:LargeRegion} for some $h>0$. Then, the multivariate
Chebyshev coefficients of the first kind for $f$ satisfy
\begin{align}\label{eq:chebybound2}
\left|a_{\mathbf{k}}^{T}\right| \leq 2^{d-\aleph(\mathbf{k})} \frac{
\widehat{\mathcal{B}}_{f}}{\rho^{\|\vec{k}\|_2}},
\end{align}
and the multivariate Chebyshev coefficients of the second kind for $f$
satisfy
\begin{align}\label{eq:chebybound3}
\left|a_{\mathbf{k}}^{U}\right| \leq \frac{\widehat{\mathcal{B}}_{f}
L(\mathcal{E}_{\pmb{\widehat{\rho}}})}{\pi^d
\rho^{\|\vec{k}\|_2+1}},
\end{align}
where $\widehat{\rho}_j = c_jh + \sqrt{1+(c_jh)^2}$ with
$c_j=k_j/\|\vec{k}\|_2$ for $j=1,\ldots,d$. 
\end{corollary}
\begin{proof}
To show \eqref{eq:chebybound3}, we note that, as in the proof of Theorem \ref{thm:akbound2}, $f$ is analytic in the polyellipse
$\mathcal{E}_{\pmb{\widehat{\rho}}} := \bigotimes_{j=1}^{d} \mathcal{E}_{{\widehat{\rho}}_j}$, where $\widehat{\rho}_j = c_jh +
\sqrt{1+(c_jh)^2}$. This, together with Corollary \ref{coro:MulChebU}, implies that
\begin{align}\label{eq:akUbound2}
\left|a_{\vec{k}}^U \right| \leq
\frac{\widehat{\mathcal{B}}_{f}L(\mathcal{E}_{\pmb{\widehat{\rho}}})}{\pi^d
\pmb{\widehat{\rho}}^{\vec{k}+1}}.
\end{align}
In view of \eqref{eq:LowBound1}, we have
\begin{align}\label{eq:Lowbound3}
\pmb{\widehat{\rho}}^{\vec{k}+1} = \prod_{j=1}^{d}
\widehat{\rho}_j^{k_j+1} \geq \prod_{j=1}^{d}
\rho^{\frac{k_j(k_j+1)}{\|\vec{k}\|_2}} = \rho^{\|\vec{k}\|_2 +
\frac{\|\vec{k}\|_1}{\|\vec{k}\|_2}} \geq \rho^{\|\vec{k}\|_2 + 1},
\end{align}
where we have made use of the fact that $\|\vec{k}\|_1\geq \|\vec{k}\|_2$
in the last step (cf. Lemma \ref{lem:NormIneq} below).
Combining \eqref{eq:Lowbound3} and \eqref{eq:akUbound2} then gives
\eqref{eq:chebybound3}.

The proof of \eqref{eq:chebybound2} is similar, where one needs to use the estimate \eqref{eq:chebybound}. We omit the details here.
This completes the proof of Corollary \ref{cor:multcheby}.
\end{proof}

\section{Multivariate Gegenbauer approximation of analytic functions}\label{sec:MultiExpRate}
In this section, we investigate the error bound of the multivariate
Gegenbauer approximation of analytic functions with the multi-indices chosen from a
specified index set.

\subsection{Multivariate Gegenbauer approximation with an $\ell^{q}$ ball index set}
We are interested in the multivariate Gegenbauer approximation corresponding to
an $\ell^{q}$ ball index set in $\mathrm{N}_0^{d}$ defined by
\begin{align}\label{eq:BallSet}
\Lambda_N^{q} = \left\{ \vec{k}\in \mathrm{N}_0^d ~\bigg{|}~\|\vec{k}\|_{q} \leq
N \right\},
\end{align}
where $q>0$ and $\|\vec{k}\|_{q}$ is defined as in \eqref{eq:Norm}.
Note that such an index set is a lower set and includes some
important index sets as special cases. For example, the total,
Euclidean and maximal degrees of a multivariate polynomial at most
$N$ correspond to $q=1,2,\infty$ in \eqref{eq:BallSet},
respectively. To gain some intuition regarding the distribution of
the grids in $\Lambda_N^{q}$, we plot in Figure \ref{fig:BallSet}
the index set $\Lambda_{30}^{q}$ for $d=2$ and three different
values of $q$.

\begin{figure}[ht]
\centering
\includegraphics[width=4.8cm,height=4.8cm]{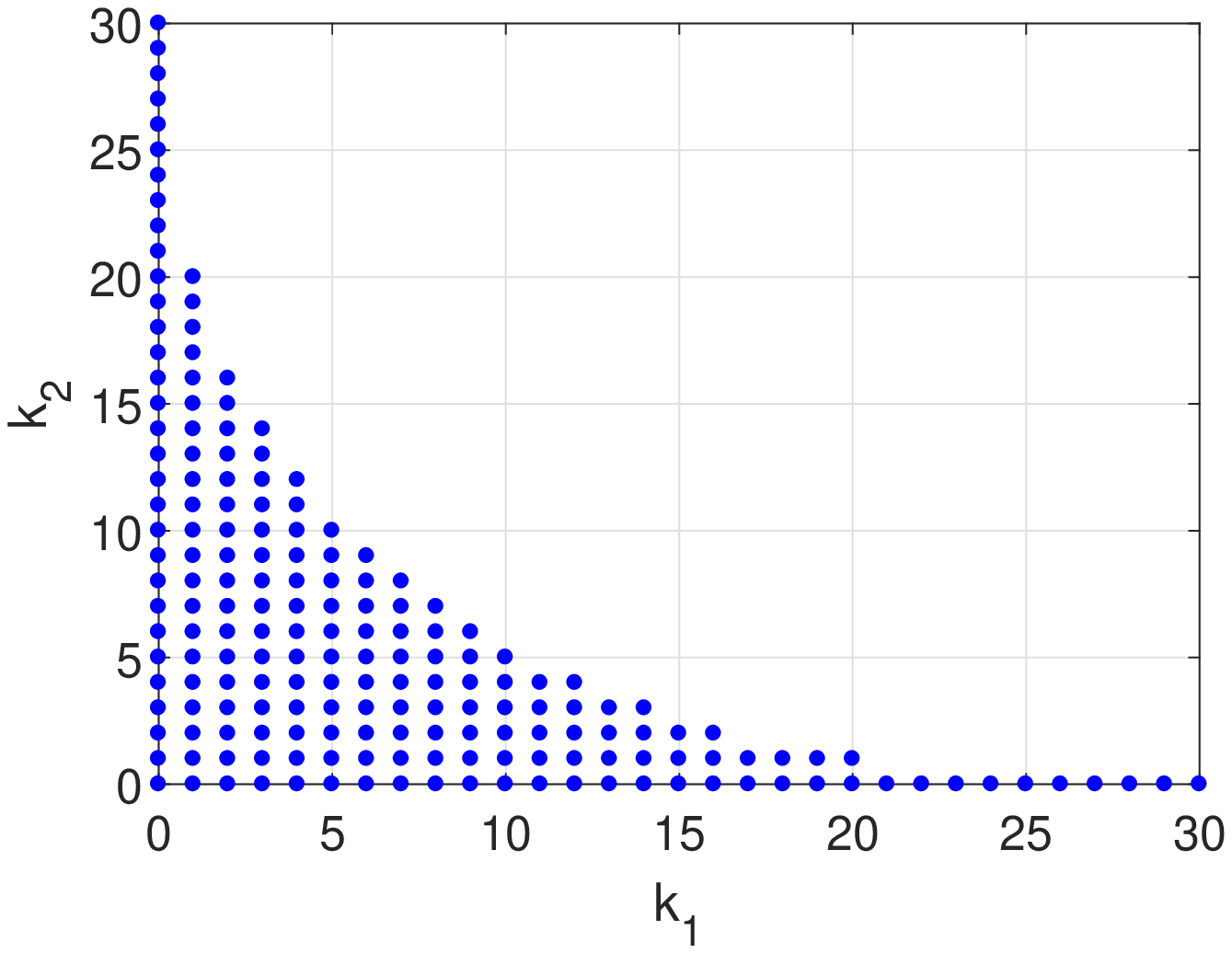}~
\includegraphics[width=4.8cm,height=4.8cm]{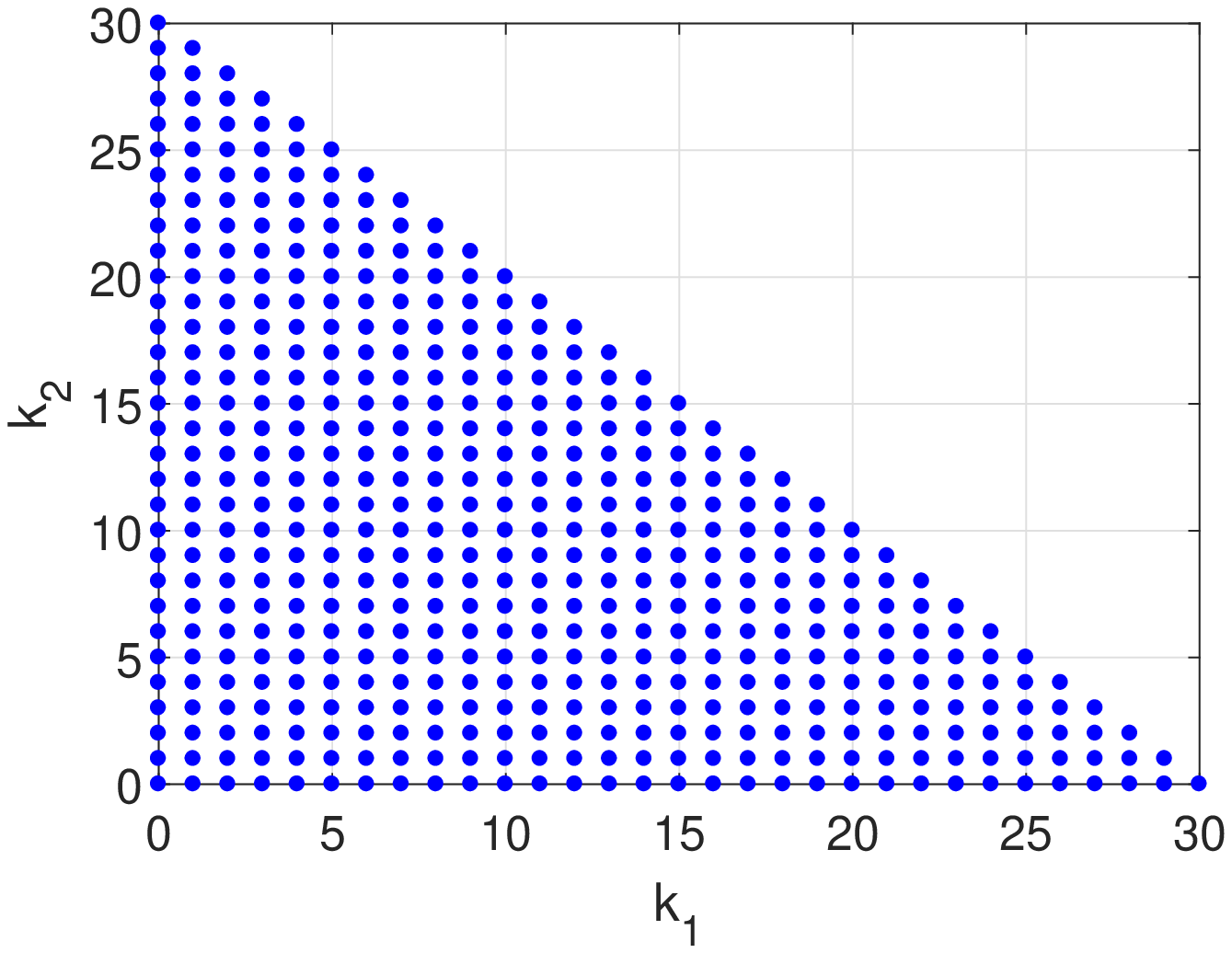}~
\includegraphics[width=4.8cm,height=4.8cm]{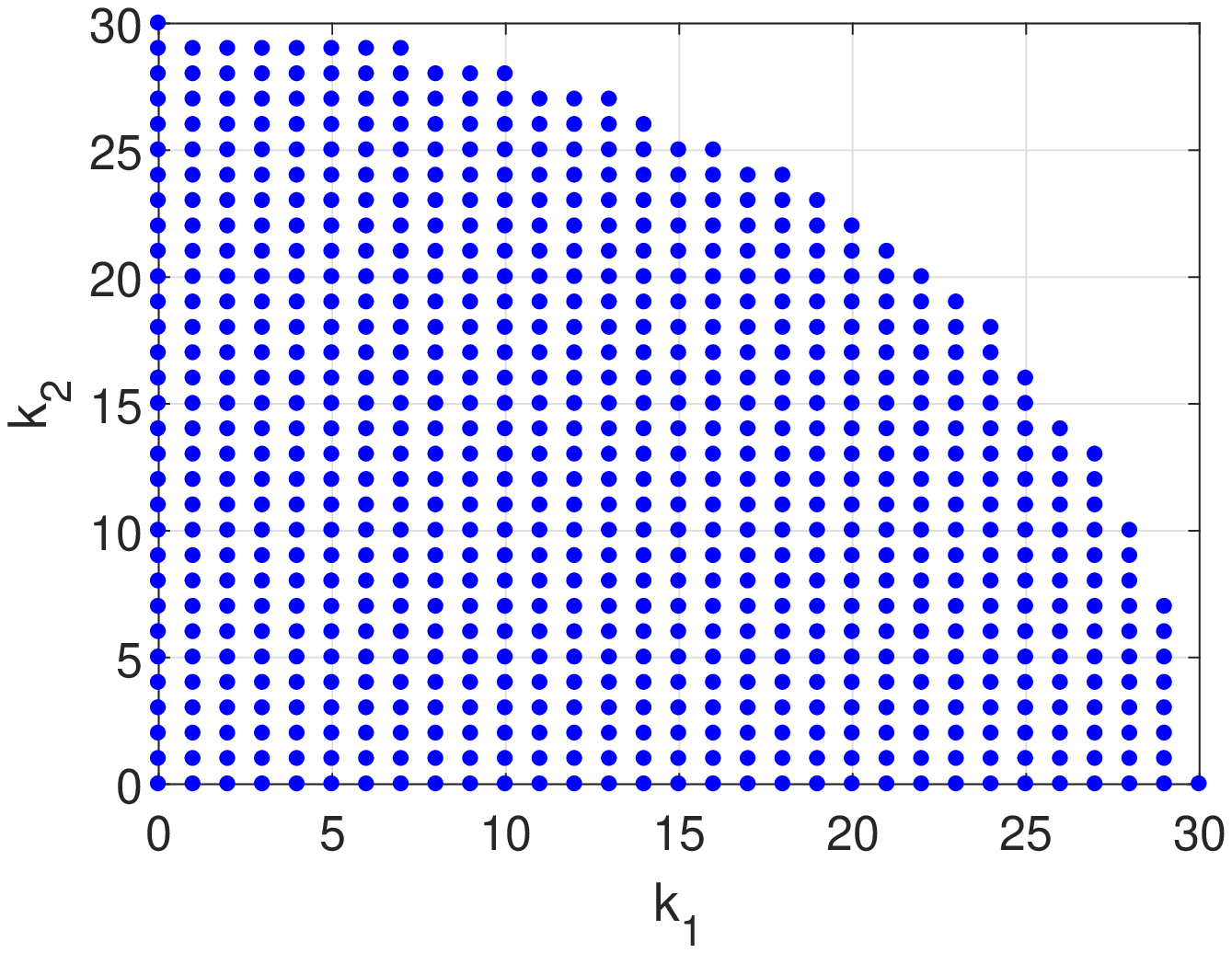}
\caption{Illustration of the index set $\Lambda_{30}^{q}$ for
$q=1/2$ (left), $q=1$ (middle) and $q=2$ (right) in dimension $d=2$.
} \label{fig:BallSet}
\end{figure}

We now consider the finite-dimensional polynomial space $\vec{P}_N^q$ corresponding to the $\ell^{q}$ ball index set, namely,
\begin{align}
\vec{P}_N^q := \spn \left \{ C_{\vec{k}}^{(\lambda)}(\mathbf{x}) ~\bigg|~ \vec{k}\in
\Lambda_N^q \right \}.
\end{align}
Let $\vec{\Pi}_N^{\lambda}$ be the orthogonal projection from the space
$L_{\omega_{\lambda}(\mathbf{x})}^2(\Omega_d)$ to  $\vec{P}_N^q$ such that
\begin{align}
\int_{\Omega_d} ((\vec{\Pi}_N^{\lambda}f)(\vec{x}) - f(\vec{x}) )
\omega_{\lambda}(\mathbf{x}) Q(\vec{x}) \mathrm{d}\vec{x} = 0,
\qquad \forall Q(\vec{x})\in \vec{P}_N^q.
\end{align}
It is well-known that $(\vec{\Pi}_N^{\lambda}f)(\vec{x})$ can be
written explicitly as
\begin{align}\label{eq:MuGeEx}
(\vec{\Pi}_N^{\lambda}f)(\vec{x}) =\left\{
                                     \begin{array}{ll}
                                      \sum_{\vec{k}\in \Lambda_N^q}
a_\vec{k} C_{\vec{k}}^{(\lambda)}(\mathbf{x}), & \hbox{$\lambda>0$,}
\\[2ex]
                                      \sum_{\vec{k}\in \Lambda_N^q}
a_\vec{k}^T T_{\vec{k}}(\mathbf{x}), & \hbox{$\lambda=0$,}
                                     \end{array}
                                   \right.
\end{align}
where the coefficients $a_\vec{k}$ and $a_\vec{k}^T$ are given in
\eqref{def:akintegral} and \eqref{eq:MulChebT}, respectively. The
main result of this section is the following theorem regarding the
explicit error bound of the multivariate Gegenbauer approximation
$(\vec{\Pi}_N^{\lambda}f)(\vec{x})$ in the uniform norm.

\begin{theorem}\label{thm:BallSet}

Let $\lambda\geq0$ and let $(\vec{\Pi}_N^{\lambda}f)(\vec{x})$
defined in \eqref{eq:MuGeEx} be the multivariate Gegenbauer
approximation associated with the index set $\Lambda_N^{q}$. Suppose
that $f$ satisfies Assumption \ref{assump2}. Then, we have,
\begin{align}\label{eq:rateGegen}
\max_{\vec{x}\in \Omega_d}
\left| f(\vec{x}) -
(\vec{\Pi}_N^{\lambda}f)(\vec{x}) \right|  \leq \mathcal{K}
\rho^{-\frac{N}{\gamma}}, \qquad N>\frac{\lambda \gamma d}{ \ln \rho},
\end{align}
where $\mathcal{K}$ is a constant independent of the index set (see \eqref{eq:kgegen} and \eqref{eq:kcheby} below for explicit representations),
$\rho=h+\sqrt{1+h^2}$ (with $h$ arising from Assumption II),
and
\begin{align}\label{eq:Gamma}
\gamma = \left\{\begin{array}{ccc}
               {\displaystyle  1  },   & \mbox{$q\geq2$, }\\ [10pt]
               {\displaystyle  d^{\frac{1}{q}-\frac{1}{2}} },    & \mbox{$0<q<2$}.
                  \end{array}
                  \right.
\end{align}

\end{theorem}

Some comments regarding Theorem \ref{thm:BallSet} are given below.
\begin{itemize}
\item Up to the algebraic pre-factor, the rate of convergence
of the multivariate Chebyshev approximation established in
\eqref{eq:rateGegen} was first obtained by Trefethen in
\cite{trefethen2017multivatiate} for $q=1,2,\infty$. Here we have
extended his result to a more general setting.

\item For the multivariate normalized Legendre approximation associated with the index
set $\Lambda_N^q$, our analysis will also lead to the same error bound as shown in \eqref{eq:rateGegen}, although the constant
$\mathcal{K}$ might be different.

\item For the bivariate Runge function
$ f(\vec{x}) = 1/(x_1^2+x_2^2+h^2), h>0, $ it is readily seen from
\cite[Theorem~3.1 and Lemma~4.6]{bos2018bernstein} that
\begin{equation}\label{eq:BerWal}
\displaystyle \limsup_{N\to\infty} D_{N,q}(f)^{1/N} \leq 1/\rho,
\end{equation}
where $q\geq2$ and
\begin{align}
D_{N,q}(f):=\inf\left\{\max_{\vec{x}\in \Omega_2} \left|f -
\sum_{\vec{k}\in \Lambda_N^{q}} c_{\vec{k}} \vec{x}^{\vec{k}}
\right|,~ c_{\vec{k}}\in \mathbb{C} \right\}. \nonumber
\end{align}

Hence, by comparing \eqref{eq:rateGegen} with \eqref{eq:BerWal} we
can conclude that the multivariate Gegenbauer and Chebyshev
approximations with an $\ell^q$ ball index set ($q\geq2$) achieve
the best possible rate of convergence of polynomial approximations in this case.
\end{itemize}


We next present the proof of Theorem \ref{thm:BallSet}, and start with some auxiliary results to be used later.

\subsection{Some auxiliary lemmas}
\begin{lemma}\label{lem:NormIneq}
Let $\|\vec{k}\|_q$ be defined in \eqref{eq:Norm}. For $r\geq
s>0$ and $\vec{k}\in \mathrm{N}_0^d$, we have
\begin{align}\label{eq:NormIneq}
\|\vec{k}\|_r \leq \|\vec{k}\|_s \leq d^{\frac{1}{s}-\frac{1}{r}}
\|\vec{k}\|_r.
\end{align}
Moreover, it is worthwhile to point out that the above inequalities are optimal in the sense
that there are no smaller constants such that they still
hold for all $\vec{k}\in \mathrm{N}_0^d$.
\end{lemma}
We note that, if $q \geq 1$, $\|\cdot \|_q$ defines a norm in $\mathbb{R}^d$ and the inequalities
\eqref{eq:NormIneq} are well-known (cf. \cite[Proposition 2.10]{wendland2018numerical}). It comes out this result can be extended to the case $q>0$, we leave the proof to the interested reader.

The second lemma is about the upper bound of an integral over an unbounded interval.

\begin{lemma}\label{lem:IntBound}
Let $a,b>0$ and $M>0$. We have
\begin{align}\label{eq:IntBound}
\int_{M}^{\infty} e^{-ax} x^b \mathrm{d}x \leq e^{-aM} \left(
\sum_{j=1}^{m_b+1} \frac{M^{b-j+1}}{a^j} \prod_{i=0}^{j-2}(b-i)
\right),
\end{align}
where $m_b$ is a positive integer depending on $b$ that is uniquely defined by
\begin{equation}\label{def:mb}
m_b=\left\{
      \begin{array}{ll}
        b, & \hbox{$b\in \mathbb{N} $,} \\[1ex]
        \lfloor b \rfloor+1, & \hbox{$b \notin \mathbb{N}$,}
      \end{array}
    \right.
\end{equation}
 and the product
in the right-hand side of \eqref{eq:IntBound} is assumed to be one
for $j<2$. In \eqref{def:mb}, $\lfloor x \rfloor$ denotes the integral part of a real number $x$.
\end{lemma}
\begin{proof}
With $m_b$ given in \eqref{def:mb}, we obtain from integration by parts $m_b$ times that
\begin{align}\label{eq:IntS1}
&\int_{M}^{\infty} e^{-ax} x^b \mathrm{d}x = -\frac{1}{a}
\int_{M}^{\infty} x^b \frac{\mathrm{d}}{\mathrm{d}x}\left(e^{-ax}\right)
= \frac{M^b}{a} e^{-aM} + \frac{b}{a} \int_{M}^{\infty} e^{-ax}
x^{b-1} \mathrm{d}x = \cdots
\nonumber \\
&= e^{-aM} \sum_{j=1}^{m_b} \frac{M^{b-j+1}}{a^j}
\prod_{i=0}^{j-2}(b-i) + \frac{1}{a^{m_b}} \prod_{i=0}^{m_b-1}(b-i)
\int_{M}^{\infty} x^{b-m_b} e^{-ax} \mathrm{d}x.
\end{align}
Note that $-1< b-m_b \leq 0$, it is easily seen that
\begin{align}\label{eq:IntS2}
\int_{M}^{\infty} x^{b-m_b} e^{-ax} \mathrm{d}x \leq M^{b-m_b}
\int_{M}^{\infty} e^{-ax} \mathrm{d}x = \frac{M^{b-m_b}}{a} e^{-aM}.
\end{align}
Combining \eqref{eq:IntS1} and \eqref{eq:IntS2} gives us the desired
result.
\end{proof}

Finally, we also need the following lemma which gives us an explicit upper bound of the ratio of Gamma functions.
\begin{lemma}\label{lemma:inequality ratio gamma}
Let $k\geq1$ and $a,b\in \mathbb{R}$. For $k+a>1$ and $k+b>1$, we
have
\begin{align}
\frac{\Gamma(k+a)}{\Gamma(k+b)} \leq \Upsilon_{k}^{a,b} k^{a-b},
\end{align}
where
\begin{align}\label{def:upsilon}
\Upsilon_{k}^{a,b} = \exp\left( \frac{a-b}{2(k+b-1)} +
\frac{1}{12(k+a-1)} + \frac{(a-1)(a-b)}{k} \right).
\end{align}
\end{lemma}
\begin{proof}
See \cite[Lemma~2.1]{zhao2013sharp}.
\end{proof}

We are now ready to prove Theorem \ref{thm:BallSet}.

\subsection{Proof of Theorem \ref{thm:BallSet}}
By \eqref{eq:gegenbauerexpansion} and \eqref{eq:MuGeEx}, it follows that, for $\lambda>0$,
\begin{align}\label{eq:max1}
\max_{\vec{x}\in \Omega_d} \left| f(\vec{x}) -
(\vec{\Pi}_N^{\lambda}f)(\vec{x}) \right|
 \leq \sum_{\vec{k}\in \mathrm{N}_0^{d}\setminus\Lambda_N^q} |
a_\vec{k}| \max_{\vec{x}\in \Omega_d}
\left|C_{\vec{k}}^{(\lambda)}(\mathbf{x})\right| = \sum_{\vec{k}\in \mathrm{N}_0^{d}\setminus\Lambda_N^q} |
a_\vec{k}| C_{\vec{k}}^{(\lambda)}(\mathbf{1}),
\end{align}
where we have made use of \eqref{eq:GegenBound} in the last step.

To this end, with $\widehat{\rho}_j$, $j=1,\ldots,d$, defined in
\eqref{eq:rhoHat}, it is readily seen that $\widehat{\rho}_j \leq
\rho_\epsilon := \sqrt{h^2+\epsilon}+\sqrt{1+h^2+\epsilon}$, and
from Assumption \ref{assump2} on $f$ and the proof of Theorem
\ref{thm:akbound2} that $\mathcal{E}_{\widehat{\pmb{\rho}}} =
\bigotimes_{j=1}^{d} \mathcal{E}_{\widehat{\rho}_j} \subseteq
D_{h,\epsilon}$. Thus, we conclude from \eqref{eq:akupbound2} that,
for any multi-index $\vec{k}\in \mathrm{N}_0^d$,
\begin{align}\label{eq:akupbound4}
\left|a_{\vec{k}}\right| &\leq
\frac{\max_{\vec{z}\in D_{h,\epsilon}}|f(\vec{z})| L(\mathcal{E}_{\rho_\epsilon})^d}{\pi^d
\rho^{\|\vec{k}\|_2}} \prod_{\substack{1\leq i \leq d \\ k_i =
0}}\overline{D}_{\widehat{\rho}_i}^{(\lambda)}
\prod_{\substack{1\leq j \leq d
\\ k_j \neq
0}} k_j^{1-\lambda} D_{\widehat{\rho}_j}^{(\lambda)},
\end{align}
where we emphasize that the constant $\displaystyle \max_{\vec{z}\in D_{h,\epsilon}}|f(\vec{z})| L(\mathcal{E}_{\rho_\epsilon})^d$ is independent of $\vec{k}$.
This, together with \eqref{eq:max1} and \eqref{eq:normalization gegenbauer}, implies that
\begin{align}\label{eq:MaxErrStepI}
\max_{\vec{x}\in \Omega_d} \left| f(\vec{x}) -
(\vec{\Pi}_N^{\lambda}f)(\vec{x}) \right|
& \leq \frac{\max_{\vec{z}\in D_{h,\epsilon}}|f(\vec{z})| L(\mathcal{E}_{\rho_\epsilon})^d}{\pi^d} \sum_{\vec{k}\in
\mathrm{N}_0^{d}\setminus\Lambda_N^q} \left(\prod_{\substack{1\leq i
\leq d \\ k_i =
0}}\overline{D}_{\widehat{\rho}_i}^{(\lambda)}\prod_{\substack{1\leq
j \leq d
\\ k_j \neq
0}} D_{\widehat{\rho}_j}^{(\lambda)} \right) \nonumber \\
&~~~ \times \left( \prod_{\substack{1\leq j\leq d\\ k_j\neq0}}
\frac{k_j^{1-\lambda}\Gamma(k_j+2\lambda)}{\Gamma(2\lambda)\Gamma(k_j+1)}
\right) \frac{1}{\rho^{||\vec{k}||_2}}.
\end{align}
For the product in the last line of the above formula, we obtain from
Lemma \ref{lemma:inequality ratio gamma} that
\begin{align}\label{eq:MaxErrStepII}
\prod_{\substack{1\leq j \leq d \\ k_j \neq 0}}
\frac{k_j^{1-\lambda}
\Gamma(k_j+2\lambda)}{\Gamma(2\lambda)\Gamma(k_j+1)} \leq
\prod_{\substack{1\leq j \leq d
\\ k_j \neq
0}} \frac{\Upsilon_{k_j}^{2\lambda,1} k_j^{\lambda}
}{\Gamma(2\lambda)} = \left( \prod_{\substack{1\leq j \leq d
\\ k_j \neq
0}} \frac{\Upsilon_{k_j}^{2\lambda,1} }{\Gamma(2\lambda)} \right)
\left( \prod_{\substack{1\leq j \leq d
\\ k_j \neq 0}} k_j \right)^{\lambda},
\end{align}
where $\Upsilon_k^{a,b}$ is defined in \eqref{def:upsilon}. A further appeal to the arithmetic geometric mean inequality shows that
\begin{align}\label{eq:AGMI}
\prod_{\substack{1\leq j \leq d \\ k_j\neq0}} k_j \leq \left(
\frac{k_1+\cdots+k_d}{d-\aleph(\mathbf{k})}
\right)^{d-\aleph(\mathbf{k})} = \left(
\frac{\|\vec{k}\|_1}{d-\aleph(\mathbf{k})}
\right)^{d-\aleph(\mathbf{k})}.
\end{align}
Thus, it follows from \eqref{eq:MaxErrStepI}, \eqref{eq:MaxErrStepII} and
\eqref{eq:AGMI} that
\begin{align}\label{eq:MaxErrStepIII}
\max_{\vec{x}\in \Omega_d} \left| f(\vec{x}) -
(\vec{\Pi}_N^{\lambda}f)(\vec{x}) \right|
&\leq
\frac{\max_{\vec{z}\in D_{h,\epsilon}}|f(\vec{z})| L(\mathcal{E}_{\rho_\epsilon})^d}{\pi^d}
\max_{\vec{k}\in \mathrm{N}_0^{d}\setminus\Lambda_N^q}
\left(\prod_{\substack{1\leq i \leq d \\ k_i =
0}}\overline{D}_{\widehat{\rho}_i}^{(\lambda)}\prod_{\substack{1\leq
j \leq d \\ k_j \neq 0}} \frac{\Upsilon_{k_j}^{2\lambda,1} D_{\widehat{\rho}_j}^{(\lambda)}}{\Gamma(2\lambda)} \right)        \nonumber \\
&~~~~~ \times \sum_{\vec{k}\in \mathrm{N}_0^{d}\setminus\Lambda_N^q}
\frac{\|\vec{k}\|_1^{\lambda d}}{\rho^{||\vec{k}||_2}},
\end{align}
where we have made use of the fact that $1 \leq d-\aleph(\mathbf{k}) \leq d $
for any $\vec{k}\in \mathrm{N}_0^{d}\setminus\Lambda_N^q$. The remaining task is then to estimate the two factors $\displaystyle \sum_{\vec{k}\in \mathrm{N}_0^{d}\setminus\Lambda_N^q}
\frac{\|\vec{k}\|_1^{\lambda d}}{\rho^{||\vec{k}||_2}}$ and
\newline
$\displaystyle \max_{\vec{k}\in \mathrm{N}_0^{d}\setminus\Lambda_N^q}
\left(\prod_{\substack{1\leq i \leq d \\ k_i =
0}}\overline{D}_{\widehat{\rho}_i}^{(\lambda)}\prod_{\substack{1\leq
j \leq d \\ k_j \neq 0}} \frac{\Upsilon_{k_j}^{2\lambda,1} D_{\widehat{\rho}_j}^{(\lambda)}}{\Gamma(2\lambda)} \right)$, respectively.

To estimate $\displaystyle \sum_{\vec{k}\in
\mathrm{N}_0^{d}\setminus\Lambda_N^q} \frac{\|\vec{k}\|_1^{\lambda
d}}{\rho^{||\vec{k}||_2}}$, we first observe from Lemma
\ref{lem:NormIneq} that $ \|\vec{k}\|_1 \leq \sqrt{d} \|\vec{k}\|_2$
and $ \|\vec{k}\|_q \leq \gamma \|\vec{k}\|_2$, where the constant
$\gamma$ depending on $d$ and $q$ is given in \eqref{eq:Gamma}.
Thus, it is readily seen that
\begin{align}
\sum_{\vec{k}\in \mathrm{N}_0^{d}\setminus\Lambda_N^q}
\frac{\|\vec{k}\|_1^{\lambda d}}{\rho^{||\vec{k}||_2}} &=
\sum_{\|\vec{k}\|_{q}>N} \frac{\|\vec{k}\|_1^{\lambda
d}}{\rho^{||\vec{k}||_2}} \leq d^{\frac{\lambda d}{2}}
\sum_{\|\vec{k}\|_{2}>\frac{N}{\gamma}} \frac{\|\vec{k}\|_2^{\lambda
d}}{\rho^{||\vec{k}||_2}}. \nonumber
\end{align}
Note that $\|\vec{k}\|_2^{\lambda d}/\rho^{||\vec{k}||_2}$ decreases
strictly for $\|\vec{k}\|_2> \lambda d/\ln\rho$, the last term can
be further bounded as
\begin{align}\label{eq:sumpart}
\sum_{\|\vec{k}\|_{2}>\frac{N}{\gamma}} \frac{\|\vec{k}\|_2^{\lambda
d}}{\rho^{||\vec{k}||_2}} &\leq
\idotsint\limits_{\substack{x_1^2+\cdots+x_d^2\geq\left(\frac{N}{\gamma}\right)^2 \\
x_1,\ldots,x_d\geq0}} \frac{(x_1^2+\cdots+x_d^2)^{\frac{\lambda
d}{2}}}{\rho^{\sqrt{x_1^2+\cdots+x_d^2}}} \mathrm{d}x_1\cdots
\mathrm{d}x_d \nonumber \\
&\leq C_d \int_{\frac{N}{\gamma}}^{\infty} \frac{t^{\lambda d + d -
1}}{\rho^t} \mathrm{d}t, \qquad N>\frac{\lambda \gamma d}{\ln \rho},
\end{align}
where we have evaluated the integral with the aid of spherical coordinates and
\begin{align} C_d = \left\{\begin{array}{cc}
                {\displaystyle 1 },      & \mbox{if $d=1$}, \\ [10pt]
                {\displaystyle
\frac{\left(\pi/2\right)^{\lfloor \frac{d}{2} \rfloor}}{(d-2)!!} },
& \mbox{if $d\geq2$}.
                  \end{array}
                  \right.
\end{align}
Next, by setting $a=\ln\rho$, $b=\lambda d+d-1$ and $M=N/\gamma$ in Lemma
\ref{lem:IntBound}, it follows that the last integral in \eqref{eq:sumpart} admits the following upper bound:
\begin{align}
\int_{\frac{N}{\gamma}}^{\infty} \frac{t^{\lambda d + d -
1}}{\rho^t} \mathrm{d}t &= \int_{\frac{N}{\gamma}}^{\infty}
e^{-t\ln\rho}
t^{\lambda d+d-1} \mathrm{d}t \nonumber
\\ &\leq  \rho^{-\frac{N}{\gamma}} \left( \sum_{j=1}^{m_{\lambda d+d-1}+1}
\frac{(\frac{N}{\gamma})^{\lambda d+d-j}}{(\ln\rho)^j}
\prod_{i=0}^{j-2}(\lambda d+d-i-1) \right), \nonumber
\end{align}
where recall that the constant $m_b$ is defined in \eqref{def:mb}. As a consequence, we finally arrive at
\begin{equation}\label{eq:est1}
\sum_{\vec{k}\in \mathrm{N}_0^{d}\setminus\Lambda_N^q}
\frac{\|\vec{k}\|_1^{\lambda d}}{\rho^{||\vec{k}||_2}} \leq
C_d \left( \sum_{j=1}^{m_{\lambda d+d-1}+1}
\frac{(\frac{N}{\gamma})^{\lambda d+d-j}}{(\ln\rho)^j}
\prod_{i=0}^{j-2}(\lambda d+d-i-1) \right) \rho^{-\frac{N}{\gamma}}.
\end{equation}

To find an upper bound of $\displaystyle \max_{\vec{k}\in \mathrm{N}_0^{d}\setminus\Lambda_N^q}
\left(\prod_{\substack{1\leq i \leq d \\ k_i =
0}}\overline{D}_{\widehat{\rho}_i}^{(\lambda)}\prod_{\substack{1\leq
j \leq d \\ k_j \neq 0}} \frac{\Upsilon_{k_j}^{2\lambda,1} D_{\widehat{\rho}_j}^{(\lambda)}}{\Gamma(2\lambda)} \right)$, on one hand, we observe from
\eqref{def:upsilon} that, for $\lambda>0$ and $k_j \geq 1$,
\begin{align}\label{eq:estUpsilon}
\Upsilon_{k_j}^{2\lambda,1}&=\exp\left(\frac{2\lambda-1}{2 k_j}+\frac{1}{12(k_j+2\lambda-1)}+\frac{(2\lambda-1)^2}{k_j}\right)
\nonumber \\
& \leq \exp\left(\max\left\{0,\frac{2\lambda-1}{2}\right\}+\frac{1}{24\lambda}+(2\lambda-1)^2\right).
\end{align}
On the other hand, in view of \eqref{eq:rhoHat}, we have
\begin{equation}\label{eq:esthatrhoj}
\widehat{\rho}_j \left\{
                   \begin{array}{ll}
                     =\sqrt{\epsilon}+\sqrt{1+\epsilon}, & \hbox{$k_j=0$,}
                     \\[1ex]
                     \geq \sqrt{\epsilon}+\sqrt{1+\epsilon}, & \hbox{$k_j \neq 0$.}
                   \end{array}
                 \right.
\end{equation}
As it can be easily seen from \eqref{eq:D} that $D_{\rho}^{(\lambda)}$ is a strictly decreasing function of $\rho > 1$ for fixed $\lambda>0$, thus, it follows from \eqref{eq:esthatrhoj} and \eqref{eq:upperboundQ} that, for $k_j \neq 0$,
\begin{align}\label{eq:estDhatrho}
D_{\widehat{\rho}_j}^{(\lambda)} \leq D_{\sqrt{\epsilon}+\sqrt{1+\epsilon}}^{(\lambda)}
&=\left\{
 \begin{array}{ll}
 \Gamma(\lambda)e^{\frac{1}{12}}\overline{D}_{\sqrt{\epsilon}+\sqrt{1+\epsilon}}^{(\lambda)}, & \hbox{$\lambda \geq 1$,}
 \\[1ex]
 \Gamma(\lambda)e^{\frac{1}{12}+\frac{1-\lambda}{2\lambda}}\overline{D}_{\sqrt{\epsilon}+\sqrt{1+\epsilon}}^{(\lambda)}, & \hbox{$0<\lambda<1$,}
                                                                                          \end{array}
                                                                                        \right.
\nonumber
\\
&\leq \Gamma(\lambda)e^{\frac{1}{12}+\max\left\{0,
\frac{1-\lambda}{2\lambda}\right\}}\overline{D}_{\sqrt{\epsilon}+\sqrt{1+\epsilon}}^{(\lambda)}.
\end{align}
Combining \eqref{eq:estUpsilon} and \eqref{eq:estDhatrho}, we obtain that
\begin{equation}\label{eq:est2}
\max_{\vec{k}\in \mathrm{N}_0^{d}\setminus\Lambda_N^q}
\left(\prod_{\substack{1\leq i \leq d \\ k_i =
0}}\overline{D}_{\widehat{\rho}_i}^{(\lambda)}\prod_{\substack{1\leq
j \leq d \\ k_j \neq 0}} \frac{\Upsilon_{k_j}^{2\lambda,1} D_{\widehat{\rho}_j}^{(\lambda)}}{\Gamma(2\lambda)} \right)
\leq \kappa^d,
\end{equation}
where
\begin{equation}
\kappa:=\max\left\{1, \frac{\Gamma{(\lambda)}}{\Gamma(2\lambda)}e^{2\max\left\{0,\frac{2\lambda-1}{2},\frac{1-\lambda}{2\lambda}\right\}+\frac{1}{24\lambda}+(2\lambda-1)^2+\frac{1}{12}} \right\}\overline{D}_{\sqrt{\epsilon}+\sqrt{1+\epsilon}}^{(\lambda)}.
\end{equation}
Substituting the estimates \eqref{eq:est1} and \eqref{eq:est2} into \eqref{eq:MaxErrStepIII} then gives us \eqref{eq:rateGegen} with
\begin{equation}\label{eq:kgegen}
\mathcal{K}=\max_{\vec{z}\in D_{h,\epsilon}}|f(\vec{z})| L(\mathcal{E}_{\rho_\epsilon})^d\left(\frac{\kappa}{\pi}\right)^dC_d \left( \sum_{j=1}^{m_{\lambda d+d-1}+1}
\frac{(\frac{N}{\gamma})^{\lambda d+d-j}}{(\ln\rho)^j}
\prod_{i=0}^{j-2}(\lambda d+d-i-1) \right).
\end{equation}

For the multivariate Chebyshev approximation of the first kind, i.e., $\lambda=0$, we note from \eqref{eq:MulChebT} and \eqref{eq:chebybound2} that
\begin{align}
\max_{\vec{x}\in \Omega_d} \left| f(\vec{x}) - (\vec{\Pi}_N^{0}f)(\vec{x}) \right |
&\leq \sum_{\vec{k}\in \mathrm{N}_0^{d}\setminus\Lambda_N^q} |
a_\vec{k}^T| \leq 2^d \max_{\vec{z}\in D_{h,0}}|f(\vec{z})|\sum_{\vec{k}\in
\mathrm{N}_0^{d}\setminus\Lambda_N^q}\frac{1}{\rho^{||\vec{k}||_2}}.
\nonumber
\end{align}
Similar to the derivation of \eqref{eq:est1}, it is readily seen that
\begin{equation}
\sum_{\vec{k}\in
\mathrm{N}_0^{d}\setminus\Lambda_N^q}\frac{1}{\rho^{||\vec{k}||_2}}\leq C_d \sum_{j=1}^{d}
\frac{(\frac{N}{\gamma})^{d-j}}{(\ln\rho)^j}\prod_{i=0}^{j-2}(d-i-1) \rho^{-\frac{N}{\gamma}}, \qquad N>0.
\end{equation}
Hence, a combination of the above two inequalities shows that, for $\lambda=0$, we still have \eqref{eq:rateGegen} but with the constant $\mathcal{K}$ replaced by
\begin{equation}\label{eq:kcheby}
\mathcal{K}= \max_{\vec{z}\in D_{h,0}}|f(\vec{z})| 2^d C_d \sum_{j=1}^{d}
\frac{(\frac{N}{\gamma})^{d-j}}{(\ln\rho)^j}\prod_{i=0}^{j-2}(d-i-1).
\end{equation}
This completes the proof of Theorem \ref{thm:BallSet}. \qed

\subsection{Numerical experiments and further discussions}\label{sec:NumerExp}
From Theorem \ref{thm:BallSet}, it is readily seen that the error
bound of the multivariate Gegenbauer approximation is
$\mathcal{O}(\rho^{-N})$ for $q \geq 2$. If $0<q<2$, the error bound
is $\mathcal{O}(\rho^{-\frac{N}{\gamma}})$, which deteriorate
gradually as $q\rightarrow0_+$. Our results match numerical
experiments very well for isotropic functions, as illustrated in
what follows.

We again consider the functions given in \eqref{def:f1} and
\eqref{def:f2}, respectively. Note that both functions satisfy
Assumption \ref{assump2} with $h^2=0.5$ and $\rho \approx
1.931851652578136$ for the former function, and $h^2=1$ and $\rho
\approx 2.414213562373095$ for the latter function. We then use
multivariate Legendre expansion (i.e., $\lambda=\frac{1}{2}$) on
$\Lambda_N^q$ to approximate these functions. In our computations,
the maximum error, i.e., $ \max_{\vec{x}\in \Omega_2}|f(\vec{x}) -
(\vec{\Pi}_N^{\frac{1}{2}}f)(\vec{x})| $, is measured by using a
finer grid in $\Omega_2$. The results are shown in Figure
\ref{fig:MaxError1} as a function of $N$ for three different
moderate values of $q$. For each $q$, we clearly observe that the
decay rate of the maximum error is consistent with the one
predicated in Theorem \ref{thm:BallSet}.

\begin{figure}[ht]
\centering
\includegraphics[width=7.4cm,height=6.8cm]{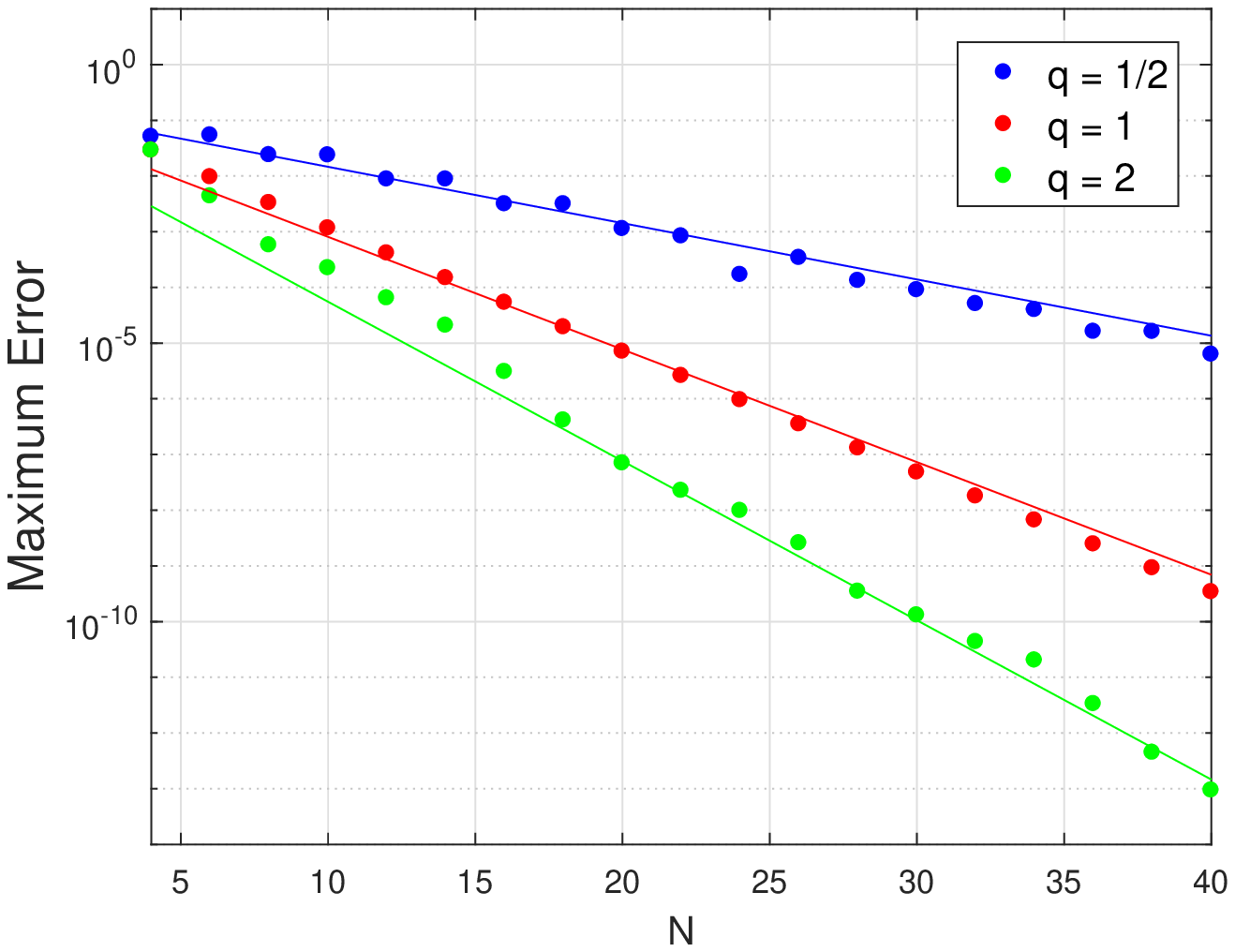}~
\includegraphics[width=7.4cm,height=6.8cm]{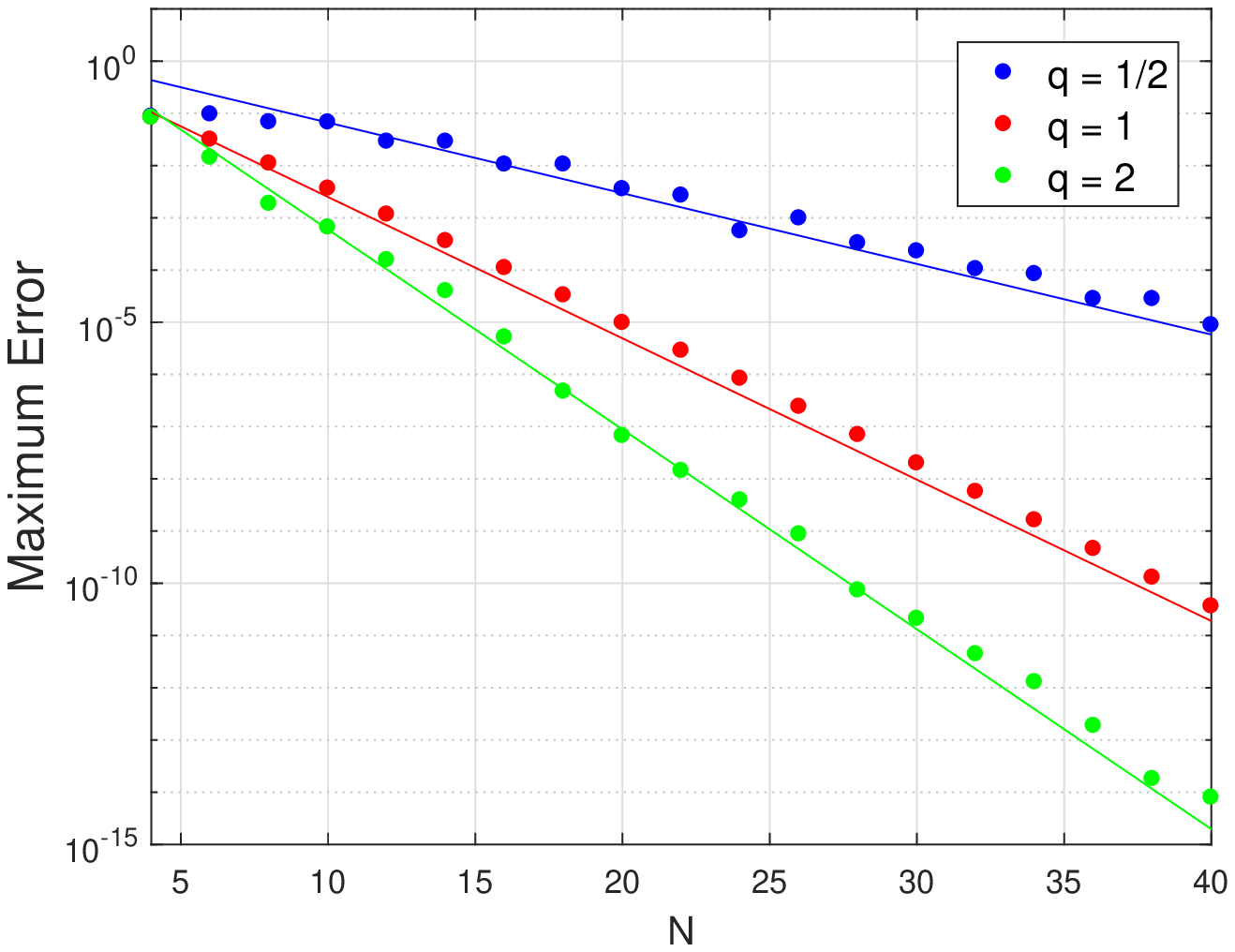}
\caption{Maximum errors of multivariate Legendre approximation for the functions \eqref{def:f1} (left)
and \eqref{def:f2} (right) as a function of $N$ in $\Omega_2$ with $q=\frac{1}{2},1,2$. Straight lines exhibit the convergence rates
predicted by Theorem \ref{thm:BallSet}.} \label{fig:MaxError1}
\end{figure}

A further numerical illustration of our results are shown in Figure
\ref{fig:MaxError2}, where we plot the maximum error of the
multivariate Legendre approximation for the function \eqref{def:f1}
with several smaller and larger values of $q$. Again, the results of
numerical experiments fit the predicted error bound in a
satisfactory way.

\begin{figure}[ht]
\centering
\includegraphics[width=7.4cm,height=6.8cm]{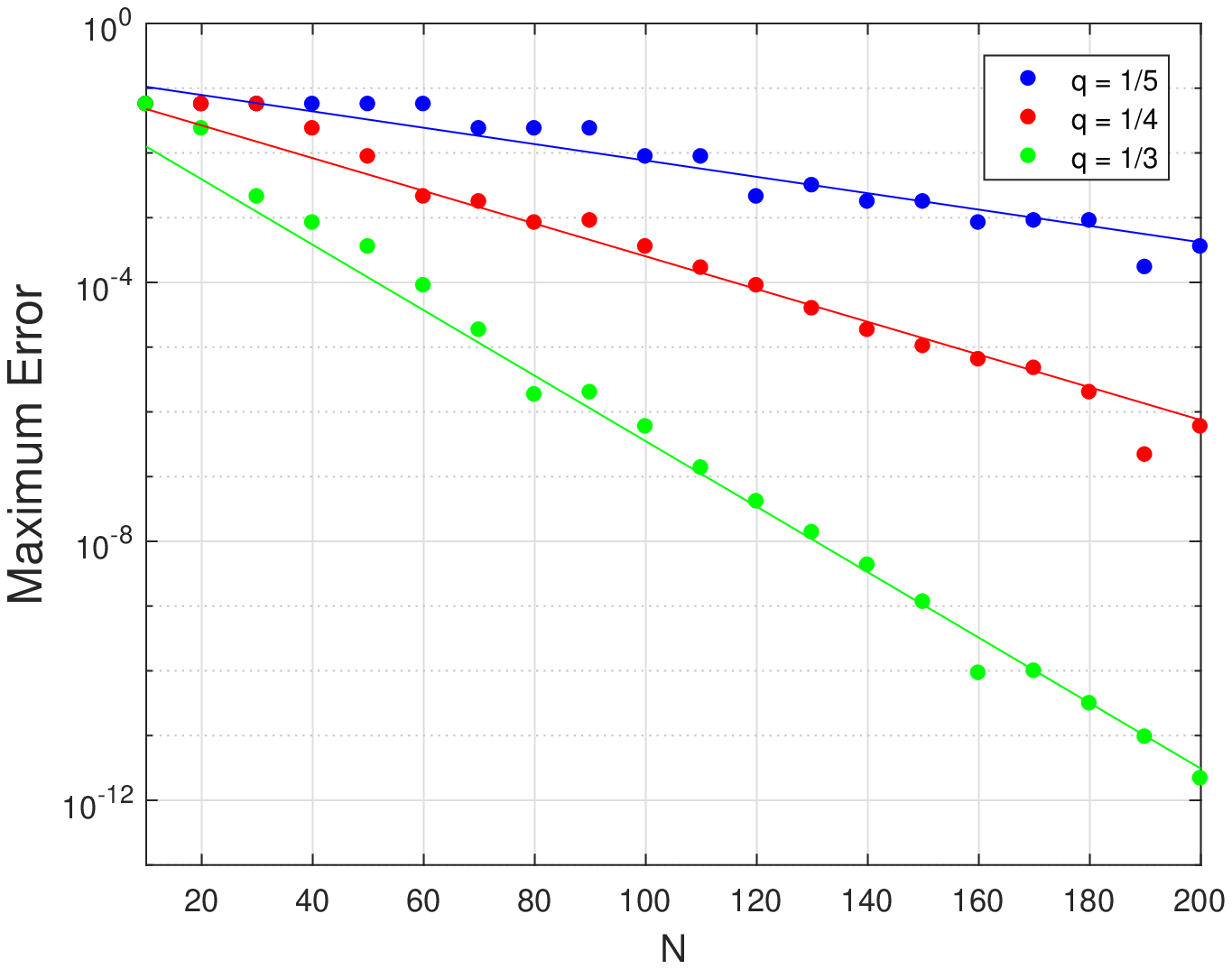}~
\includegraphics[width=7.4cm,height=6.8cm]{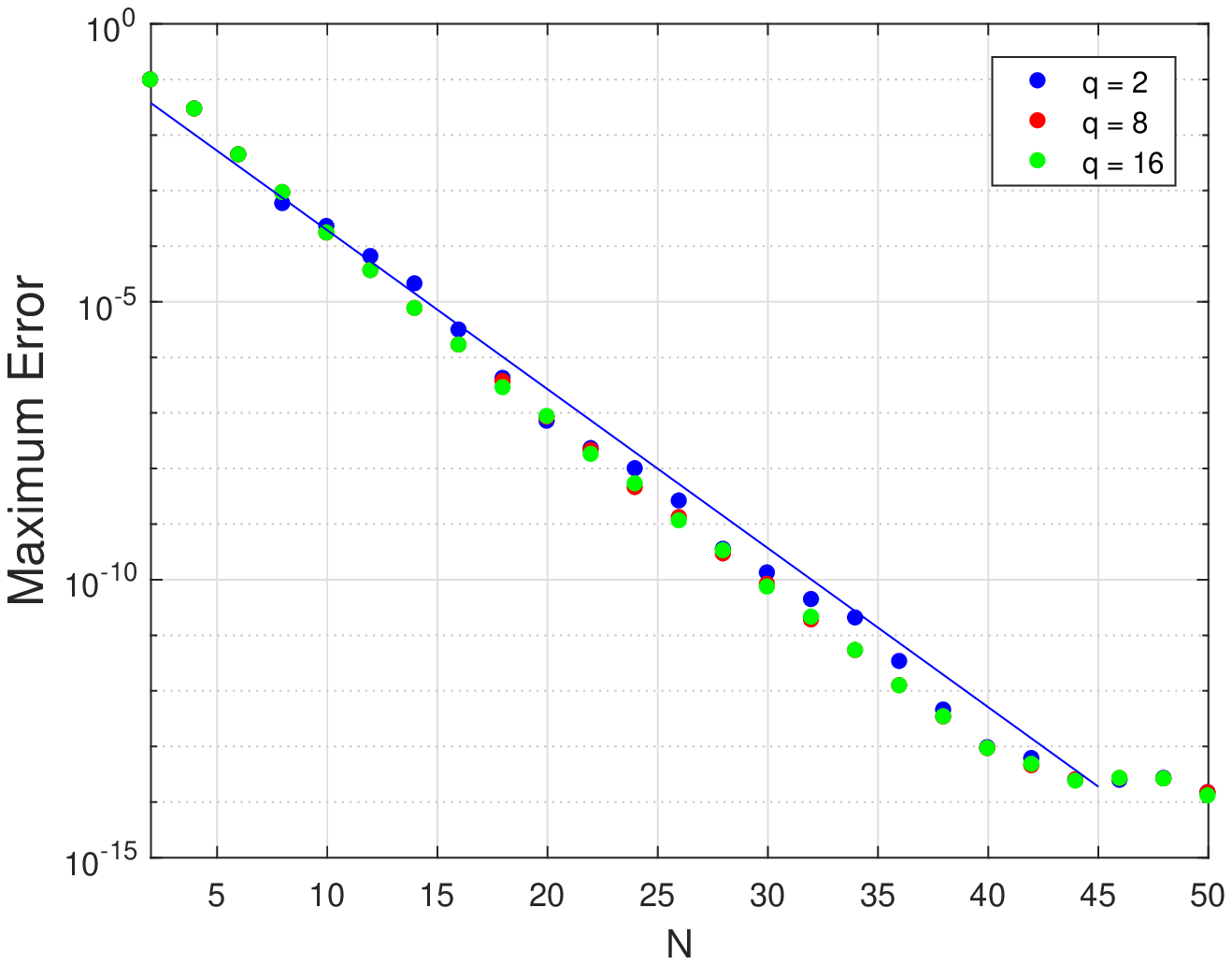}
\caption{Maximum errors of multivariate Legendre approximation of \eqref{def:f1} as a
function of $N$ in $\Omega_2$, for
$q=\frac{1}{5},\frac{1}{4},\frac{1}{3}$ (left) and $q=2,8,16$
(right). Straight lines exhibit the convergence rates predicted by
Theorem \ref{thm:BallSet}.} \label{fig:MaxError2}
\end{figure}

Finally, it is worthwhile to point out that a direct comparison of
the rates of convergence of $(\vec{\Pi}_N^{\lambda}f)(\vec{x})$
established in Theorem \ref{thm:BallSet} for different $q$ is not
fair, since the number of terms in
$(\vec{\Pi}_N^{\lambda}f)(\vec{x})$ also depends on $q$. Indeed, for
large $N$, we could estimate this number denoted by $N_q$ via a
continuum approximation and obtain that
\begin{equation}\label{def:nq}
 N_{q} \approx N^d V_q,
\end{equation}
where
$
V_q = \frac{\Gamma(\frac{1}{q}+1)^d}{\Gamma(\frac{d}{q}+1)}, q>0
$,
is the volume of the unit $\ell^q$ ball restricted to the positive
orthant (cf. \cite{wang2018volume} and the references therein).
Thus, to evaluate the efficiency of multivariate Gegenbauer
approximation with two different $\ell^q$ index sets, it is much
more reasonable to compare $N_q$ under the condition that the same
convergence rate is achieved. Assume that the predicted rate of
convergence of $(\vec{\Pi}_N^{\lambda}f)(\vec{x})$ given in
\eqref{eq:rateGegen} is
sharp, we 
compare two different $\ell^q$ ball index sets: (i) $q\neq2$ and
(ii) $q=2$. If $q\in(0,2)$, to achieve the same convergence rate,
say, $\mathcal{O}(\rho^{-N})$, it follows from \eqref{eq:rateGegen}
that the number of terms in $(\vec{\Pi}_N^{\lambda}f)(\vec{x})$
corresponding to the set (i) is equal to the number of the
multi-indices satisfying $\|\vec{k}\|_q \leq
d^{\frac{1}{q}-\frac{1}{2}}N$, while that corresponding to the set
(ii) is equal to the number of the multi-indices satisfying
$\|\vec{k}\|_2 \leq N$. By \eqref{def:nq}, it is easily seen that
the ratio of these two numbers admits the following estimate:
\begin{align}\label{eq:Ratio}
 \frac{(d^{\frac{1}{q}-\frac{1}{2}}N)^d
\frac{\Gamma(\frac{1}{q}+1)^d}{\Gamma(\frac{d}{q}+1)}}{N^d
\frac{\Gamma(\frac{3}{2})^d}{\Gamma(\frac{d}{2}+1)}} =
d^{\frac{d}{q}-\frac{d}{2}}
\left(\frac{\Gamma(\frac{1}{q}+1)}{\Gamma(\frac{3}{2})}\right)^d
\frac{\Gamma(\frac{d}{2}+1)}{\Gamma(\frac{d}{q}+1)}.
\end{align}
Numerical experiments show that the above ratio is always greater
than one for $d\geq2$ and grows exponentially fast as $d$ increases.
This means that the index set induced from an $\ell^q$ ball with
$q\in(0,2)$ may be less efficient compared with $q=2$. If $q>2$,
from \eqref{eq:rateGegen} we see that the predicted rate of
convergence is always the same, and one only needs to compare $N_q$
and $N_2$. By \eqref{def:nq}, it is easily seen that $N_q$ is
strictly increasing as $q$ increases and thus $N_q
> N_2$ for $q>2$. As a consequence, we conclude that the multivariate Gegenbauer
approximation based on the Euclidean degree of multivariate
polynomial, i.e., on the index set $\Lambda_N^q$ with $q=2$,
provides an optimal choice among the multivariate Gegenbauer
approximation with an $\ell^q$ ball index set, if the convergence
rate of $(\vec{\Pi}_N^{\lambda}f)(\vec{x})$ established in Theorem
\ref{thm:BallSet} is sharp. For the multivariate Chebyshev
approximation with an $\ell^q$ ball index set and $q=1,2,\infty$,
this viewpoint was first proposed by Trefethen in
\cite{trefethen2017cubature}. Here, we have extended his conclusion
to a general setting.


\section{Multivariate Gegenbauer approximation of functions with finite regularity}\label{sec:diff}
In this section, we give an attempt to consider multivariate
Gegenbauer approximation of functions with finite regularity. More
precisely, we restrict our discussion on the Sobolev space
$\mathcal{H}^{\vec{m}}(\Omega_d)$ defined by
\begin{align}
\mathcal{H}^{\vec{m}}(\Omega_d): = \big\{f~\big|~
\partial^{\vec{n}} f \in C(\Omega_d), ~\vec{0}\leq \vec{n} \leq
\vec{m}-1,~ \partial^{\vec{m}} f\in L^2(\Omega_d)  \big\},
\end{align}
where $\vec{m}=(m_1,\ldots,m_d)$ is a
fixed multi-index with $m_j \geq 1$, $j=1,\ldots,d$,
and the mixed derivatives of $f$ defined in \eqref{def:derivative} are understood in the
distributional sense. As in the case of analytic functions, we start with the estimate of the expansion coefficients.

\begin{theorem}\label{thm:FRBound}
Suppose $\lambda>0$ and $f \in \mathcal{H}^{\vec{m}}(\Omega_d)$, the
multivariate Gegenbauer coefficients of $f$ satisfy
\begin{align}
|a_{\mathbf{k}}| \leq \frac{V_{\vec{k},\vec{m}}}{h_{\vec{k}}^{(\lambda)}}
\prod_{j=1}^{d} \left[\sqrt{h_{k_j-\min\{k_j,m_j\}}^{\lambda+\min\{k_j,m_j\}}}\prod_{s=0}^{\min\{k_j,m_j\}-1}
\frac{2(\lambda+s)}{(k_j-s)(k_j+2\lambda+s)} \right],
\end{align}
where
\begin{equation}\label{def:Vkm}
V_{\vec{k},\vec{m}}=\sqrt{\int_{\Omega_d} |(\partial^{\min\{\vec{k},\vec{m}\}}
f)(\vec{x})|^2 \omega_{\lambda+\min\{\vec{k},\vec{m}\}}(\vec{x})
\ud \vec{x}}
\end{equation}
with $\min\{\vec{k},\vec{m}\}$ defined in \eqref{def:minindices}, and the factor depending on $s$ is taken to be 1 if $s<0$.
\end{theorem}
\begin{proof}
From \cite[Equation (18.9.20)]{DLMF}, we know that
\begin{align}
\omega_{\lambda}(x) C_n^{(\lambda)}(x) =
-\frac{2\lambda}{n(n+2\lambda)} \frac{\mathrm{d}}{\mathrm{d}x}\left[
\omega_{\lambda+1}(x) C_{n-1}^{(\lambda+1)}(x) \right],\qquad n\geq1.
\nonumber
\end{align}
Substituting this formula into \eqref{def:akintegral} and integrating by
parts $\min\{k_j,m_j\}$ times with respect to each $x_j$, we obtain
\begin{multline}\label{eq:FRStep1}
a_{\vec{k}} = \frac{1}{h_{\vec{k}}^{(\lambda)}} \prod_{j=1}^{d}
\left[\prod_{s=0}^{\min\{k_j,m_j\}-1}
\frac{2(\lambda+s)}{(k_j-s)(k_j+2\lambda+s)} \right] \int_{\Omega_d}
(\partial^{\min\{\vec{k},\vec{m}\}} f)(\vec{x}) \omega_{\lambda+\min\{\vec{k},\vec{m}\}}(\vec{x})
\\ \times
C_{\vec{k}-\min\{\vec{k},\vec{m}\}}^{\lambda+\min\{\vec{k},\vec{m}\}}(\vec{x}) \mathrm{d}\vec{x},
\end{multline}
where $C_{\vec{k}-\min\{\vec{k},\vec{m}\}}^{\lambda+\min\{\vec{k},\vec{m}\}}(\vec{x}) =
\prod_{j=1}^{d}C_{k_j-\min\{k_j,m_j\}}^{\lambda+\min\{k_j,m_j\}}(x_j)$.
By Cauchy-Schwarz
inequality, it is then readily seen from \eqref{eq:FRStep1} that
\begin{align}\label{eq:FRStep2}
|a_{\vec{k}}| &\leq  \frac{V_{\vec{k},\vec{m}}}{h_{\vec{k}}^{(\lambda)}} \prod_{j=1}^{d}
\left[\prod_{s=0}^{\min\{k_j,m_j\}-1}
\frac{2(\lambda+s)}{(k_j-s)(k_j+2\lambda+s)} \right]
\nonumber \\
&~~~ \times \sqrt{\int_{\Omega_d}
\left(C_{\vec{k}-\min\{\vec{k},\vec{m}\}}^{\lambda+\min\{\vec{k},\vec{m}\}}(\vec{x})\right)^2
\omega_{\lambda+\min\{\vec{k},\vec{m}\}}(\vec{x}) \mathrm{d}\vec{x}} \nonumber \\
&= \frac{V_{\vec{k},\vec{m}}}{h_{\vec{k}}^{(\lambda)}}
\prod_{j=1}^{d} \left[\sqrt{h_{k_j-\min\{k_j,m_j\}}^{\lambda+\min\{k_j,m_j\}}}\prod_{s=0}^{\min\{k_j,m_j\}-1}
\frac{2(\lambda+s)}{(k_j-s)(k_j+2\lambda+s)} \right],  \nonumber
\end{align}
where $V_{\vec{k},\vec{m}}$ is defined in \eqref{def:Vkm}.
This completes the proof of Theorem \ref{thm:FRBound}.
\end{proof}

As an application of the above theorem, we are able to establish an
$L^{\infty}$ error estimate of the multivariate Gegenbauer
approximation of functions with finite regularity. For simplicity of
the presentation, we restrict our attention to the ball index set
$\Lambda_N^{\infty}$, i.e., on the full gird.
\begin{theorem}
Let $\lambda>0$ and $\vec{m}=(m_1,\ldots,m_d)$ be a fixed
multi-index with $m_j>\lambda+1$ for $j=1,\ldots,d$. Suppose $f \in
\mathcal{H}^{\vec{m}}(\Omega_d)$, then we have
\begin{equation}\label{eq:estFR}
\max_{\vec{x}\in \Omega_d} \left| f(\vec{x}) -
 \sum_{\vec{k}\in \Lambda_N^\infty }
a_\vec{k} C_{\vec{k}}^{(\lambda)}(\mathbf{x})\right| \leq
\widehat{\mathcal{K}}N^{-\min_{j=1,\ldots,d}\{m_j-\lambda-1\}}, \qquad N\gg1,
\end{equation}
for some constant $\widehat{\mathcal{K}}$ independent of $N$ and $\vec{m}$.
\end{theorem}
\begin{proof}
If $f \in \mathcal{H}^{\vec{m}}(\Omega_d)$, one can check from
Theorem \ref{thm:FRBound}, \eqref{eq:normalization constant} and Lemma \ref{lemma:inequality
ratio gamma} that
\[
|a_{\vec{k}}| \leq \widehat{\mathcal{K}}_0 \prod_{j=1}^d k_j^{-\lambda-\min\{k_j,m_j\}+1},
\]
for some constant $\widehat{\mathcal{K}}_0$ independent of $\vec{k}$, where $k_j^{-\lambda-\min\{k_j,m_j\}+1}$ should be replaced by 1
if $k_j=0$. Also note that $C_k^{(\lambda)}(1)\leq C_\lambda k^{2\lambda-1}$ for some constant $C_\lambda$ independent of $k$ (see \eqref{eq:normalization gegenbauer} and
Lemma \ref{lemma:inequality ratio gamma}), it is then readily seen that
\begin{align}\label{eq:S1}
 \left| f(\vec{x}) -
 \sum_{\vec{k}\in \Lambda_N^\infty }
a_\vec{k} C_{\vec{k}}^{(\lambda)}(\mathbf{x})\right| & \leq
\sum_{\vec{k}\in \mathrm{N}_0^{d}\setminus\Lambda_N^{\infty}} |
a_\vec{k}| C_{\vec{k}}^{(\lambda)}(\mathbf{1}) \leq
\widehat{\mathcal{K}}_0C_\lambda^d \sum_{\vec{k}\in
\mathrm{N}_0^{d}\setminus\Lambda_N^{\infty}}
\prod_{j=1}^d k_j^{\lambda-\min\{k_j,m_j\}}  \nonumber \\
&= \widehat{\mathcal{K}}_0C_\lambda^d \sum_{i=1}^d
\sum_{\Xi(\mathbf{k})=i} \prod_{j=1}^d
k_j^{\lambda-\min\{k_j,m_j\}},
\end{align}
where $\Xi(\mathbf{k}):=\#\{i:k_i>N\}$.

If $\Xi(\mathbf{k})=1$, it is easy
to see that
$
\Xi(\mathbf{k})=1 \Leftrightarrow \vec{k}\in \bigcup_{l=1}^{d}
\Lambda_N^l,
$
where
$\Lambda_N^l=\{\vec{k}\in\mathbb{N}_0^d~|~k_l>N,~k_j\leq
N, j\neq l \}$. The estimate of the sum in \eqref{eq:S1} over the multi-index set $\Xi(\mathbf{k})=1$ then reduces to the estimate of the sum over each $\Lambda_N^l$. A straightforward calculation shows that, for $l=1,\ldots,d$ and $N\gg1$,
\begin{align}\label{eq:FiniteS2}
\sum_{\vec{k}\in \Lambda_N^{l}}
\prod_{j=1}^dk_j^{\lambda-\min\{k_j,m_j\}} &= \left( \prod_{j=1, j
\neq l}^{d} \sum_{k_{j}=0}^{N} k_{j}^{\lambda-\min\{k_{j},m_{j}\}})
\right) \left( \sum_{k_{l}>N}^{\infty}
k_{l}^{\lambda-\min\{k_{l},m_{l}\}} \right)  \nonumber \\
&= \left( \prod_{j=1, j \neq l}^{d} \sum_{k_{j}=0}^{N}
k_{j}^{\lambda-\min\{k_{j},m_{j}\}}) \right) \left(
\sum_{k_{l}>N}^{\infty} k_{l}^{\lambda-m_l} \right) \nonumber \\
&\leq \widehat{\mathcal{K}}_l N^{\lambda-m_l+1},
\end{align}
for some constant $\widehat{\mathcal{K}}_l$ independent of $N$ and $\vec{m}$.
Therefore, we conclude that the sum over $\Xi(\mathbf{k})=1$ can
be bounded by $\widehat{\mathcal{K}}N^{-\min_{j=1,\ldots,d}\{m_j-\lambda-1\}}$
for some constant $\widehat{\mathcal{K}}$ independent of $N$ and $\vec{m}$.

If $\Xi(\mathbf{k})=i>1$, one can show in a similar manner that the sum in \eqref{eq:S1} over this multi-index set will contribute to a worse decay than $O(N^{-\min_{j=1,\ldots,d}\{m_j-\lambda-1\}})$, which finally leads to the estimate \eqref{eq:estFR}.
\end{proof}


\section{An extension to polynomial approximation of parameterized PDEs}
\label{sec:appinPDE}

In this section, we will apply an extension of Theorem
\ref{thm:akbound1} with emphasis on tensorized Legendre expansions
to the polynomial approximation for parameterized PDEs.

The extension deals with a function $f(x,\mathbf{y})$ defined in a
bounded regular domain $D \subset \mathbb{R}^n $ with the parameters
$\mathbf{y} \in \Omega_d$. Suppose that $f\in L^{\infty}(\Omega_d,
V, \Pi_{i=1}^d \ud x_i)$, where $V=V(D)$ is certain Banach space
equipped with the norm $||\cdot||_{V(D)}$. Then, $f$ admits the
following tensorized Legendre expansions
\begin{equation}\label{eq:fxyexp}
f(x,\mathbf{y})= \sum_{\vec{k}\in\mathbb{N}_0^{d}} a_{\vec{k}}(x)
P_{\vec{k}}(\mathbf{y})= \sum_{\vec{k}\in\mathbb{N}_0^{d}}
\overline{a}_{\vec{k}}(x) \overline{P}_{\vec{k}}(\mathbf{y}),
\end{equation}
where the convergence is understood in $L^{2}(\Omega_d, V,
\Pi_{i=1}^d \ud x_i)$, and, as in \eqref{def:tensorLeg}, we have
\begin{equation}
a_{\mathbf{k}}(x) = \frac{1}{h_{\vec{k}}^{(\frac{1}{2})}}
\int_{\Omega_d}  f(x,\mathbf{y}) P_{\mathbf{k}}(\mathbf{y})
\mathrm{d}\mathbf{y}, \qquad \overline{a}_{\mathbf{k}}(x) =
\frac{a_{\mathbf{k}}(x)}{\left(\vec{k}+\frac{1}{2}\right)^{\frac{1}{2}}}.
\end{equation}
By assuming that the dependence of the parameters $\mathbf{y}$ is
analytically smooth, we have the following estimates of the
coefficients $a_{\mathbf{k}}(x)$ and $\overline{a}_{\mathbf{k}}(x)$.
\begin{proposition}\label{prop:extLegcoeff}
Let $f(x,\mathbf{y})$ be a function defined in a bounded regular
domain $D \subset \mathbb{R}^n $ with the parameters $\mathbf{y} \in
\Omega_d$. Suppose that $f\in L^{\infty}(\Omega_d, V, \Pi_{i=1}^d
\ud x_i)$, where $V=V(D)$ is certain Banach space equipped with the
norm $||\cdot||_{V(D)}$, and the analytic continuation
$f(x,\mathbf{z})$ of $f(x,\mathbf{y})$ satisfies Assumption I, we
have the following estimates of the coefficients in
\eqref{eq:fxyexp}:
\begin{align}\label{eq:LegCoeffgen}
\| a_{\mathbf{k}} \|_{V(D)} &\leq \frac{ \sup_{\mathbf{z}\in
\mathcal{E}_{\pmb{\rho}}}\| f \|_{V(D)}
L(\mathcal{E}_{\pmb{\rho}})}{\pi^d \pmb{\rho}^{\vec{k}}}
\prod_{\substack{1\leq i \leq d
\\ k_i = 0}}\overline{D}_{\rho_i}^{(\frac{1}{2})}
\prod_{\substack{1\leq j \leq d \\ k_j \neq 0}} \sqrt{k_j}
D_{\rho_j}^{(\frac{1}{2})},
\end{align}
and
\begin{align}\label{eq:NormLegCoeffgen}
\|\overline{a}_{\mathbf{k}} \|_{V(D)} &\leq
\frac{2^{\frac{\aleph(\mathbf{k})}{2}} \sup_{\mathbf{z}\in
\mathcal{E}_{\pmb{\rho}}}\| f
\|_{V(D)}L(\mathcal{E}_{\pmb{\rho}})}{\pi^d \pmb{\rho}^{\vec{k}}}
\prod_{\substack{1\leq i \leq d
\\ k_i = 0}}\overline{D}_{\rho_i}^{(\frac{1}{2})}
\prod_{\substack{1\leq j \leq d \\ k_j \neq 0}}
D_{\rho_j}^{(\frac{1}{2})},
\end{align}
where the constants $\overline{D}_{\rho_i}^{(\frac{1}{2})}$,
$D_{\rho_j}^{(\frac{1}{2})}$ are defined in \eqref{eq:upperboundQ}
and \eqref{eq:D}, respectively.
\end{proposition}
\begin{proof}Since the proof is similar to that of Theorem \ref{thm:akbound1}, we only sketch the proof of \eqref{eq:LegCoeffgen}. Thanks to the analytic dependence of $\mathbf{y}$, as in the derivation of \eqref{eq:coutour}, we obtain from Cauchy's integral formula that
$
a_{\vec{k}}(x) = \left( \frac{1}{\pi i} \right)^d
\oint_{\mathcal{E}_{\pmb{\rho}}} f(x,\mathbf{z})
\mathcal{Q}_{\mathbf{k}}^{(\frac12)}(\mathbf{z})
\mathrm{d}\mathbf{z}
$.
Hence, it is readily seen that
\begin{align}
\| a_{\mathbf{k}} \|_{V(D)}
\leq \frac{ \sup_{\mathbf{z}\in \mathcal{E}_{\pmb{\rho}}}\| f
\|_{V(D)} L(\mathcal{E}_{\pmb{\rho}})}{\pi^d} \prod_{i=1}^d
\max_{z_i \in \mathcal{E}_{\rho_i}} \left|
Q_{k_i}^{(\frac12)}(z_i)\right|.
\end{align}
This, together with Proposition \ref{prop:upperbound}, gives us
\eqref{eq:LegCoeffgen}.

This completes the proof of Proposition \ref{prop:extLegcoeff}.
\end{proof}

As an application of the above proposition, let us consider a family
of elliptic PDEs of the form
\begin{align}
\label{eq:PDE} \left\{
       \begin{array}{ll}
        -\nabla \cdot (a(x,\mathbf{y})\nabla u(x,\mathbf{y})) = f(x), &~~ \hbox{$\forall (x,\mathbf{y})\in D\times \Gamma$,}\\[2ex]
        u(x,\mathbf{y}) =0, &~~ \hbox{$\forall (x,\mathbf{y})\in \partial D\times \Gamma$,}
       \end{array}
        \right.
\end{align}
where $D \subset \mathbb{R}^n$ is a bounded Lipschitz domain with
$n\in\{1,2,3\}$, the diffusion coefficient $a(x,\mathbf{y})$ is a
function of $x$ and of parameters $\mathbf{y}=\{y_1,\ldots,y_d\} \in
\Gamma=\Omega_{d}$, and the function $f$ is a fixed function on $D$.
The gradient operator $\nabla$ is taken with respect to $x$. It is
assumed that $a$ and $f$ are chosen such that the system
\eqref{eq:PDE} is well-defined in the Sobolev space $V(D):=H_0^1(D)$
equipped with the energy norm $\|\cdot\|_{V(D)}:=\|\nabla(\cdot)
\|_{L^2(D)}$. Parameterized linear elliptic PDEs of this type arise
in a variety of stochastic and deterministic modeling of complex
systems; cf. \cite{KleiberBook,Milani2008}.

Since the solution of \eqref{eq:PDE} depends smoothly on the
coefficient $a$, a major method to find it is based on a polynomial
approximation, which leads to an approximation to the solution $u$
of the form
\begin{align}\label{eq:MultiAppr}
u_{\Lambda}(x,\mathbf{y}) = \sum_{\vec{k}\in \Lambda} c_{\vec{k}}(x)
\Psi_{\vec{k}}(\mathbf{y}),
\end{align}
where $\Lambda\in \mathbb{N}_0^{d}$ is a finite index set,
$\Psi_{\vec{k}}(\mathbf{y})$ is a multivariate polynomial, and
$c_{\vec{k}}(x) \in V(D)$ is the coefficient to be computed. Suppose
that $\|\Psi_{\vec{k}}(\mathbf{y})\|_{L^{\infty}(\Gamma)}=1$, the
error of the approximation \eqref{eq:MultiAppr} can be bounded by
\begin{align}\label{eq:truncerror}
\sup_{\mathbf{y}\in \Gamma}\bigg\|u(x,\mathbf{y}) -
u_{\Lambda}(x,\mathbf{y}) \bigg\|_{V(D)} \leq \sum_{\vec{k}\in
\Lambda^c} \|c_{\vec{k}}(x)\|_{V(D)},
\end{align}
where $\Lambda^c$ denotes the complement of $\Lambda$ in
$\mathbb{N}_0^{d}$.

In practice, the polynomials $\Psi_{\vec{k}}(\mathbf{y})$ are often
chosen to be the monomials or the tensorized Legendre polynomials
(cf. \cite{beck2014convergence,cohen2011analytic,tran2017analysis}),
which correspond to Taylor and Legendre approxiamtions,
respectively. For the latter case, both the Legendre and the
normalized Legendre expansions, i.e.,
\begin{align}\label{eq:tensorLeg}
u(x,\mathbf{y}) =  \sum_{\vec{k}\in\mathbb{N}_0^{d}} u_{\vec{k}}(x)
P_{\vec{k}}(\mathbf{y}) =  \sum_{\vec{k}\in\mathbb{N}_0^{d}}
v_{\vec{k}}(x) \overline{P}_{\vec{k}}(\mathbf{y}),
\end{align}
have been discussed. In view of the truncation error given in
\eqref{eq:truncerror}, an effective way of computation requires
using the multi-index set largest the norm of the coefficients among
all the multi-index sets with fixed cardinality. This is usually a
difficult task in implementation. Alternatively, one could relax the
condition by performing the so-called quasi-optimal approximation,
that is, the multi-index set is chosen so that the upper bounds of
the coefficients are maximized. A general strategy for convergence
analysis of quasi-optimal polynomial approximations for
parameterized PDEs \eqref{eq:PDE} was presented in
\cite{tran2017analysis}, and a key ingredient of the analysis
therein is the upper bounds of the Legendre coefficients
$u_{\vec{k}}(x)$ and $v_{\vec{k}}(x)$ given in \eqref{eq:tensorLeg}.
In what follows, we will provide sharper bounds of $u_{\vec{k}}(x)$
and $v_{\vec{k}}(x)$ with the aid of Proposition
\ref{prop:extLegcoeff}, which improve those used in
\cite{tran2017analysis}; see also
\cite{beck2014convergence,cohen2011analytic}.

Following the framework proposed in \cite{tran2017analysis}, we make
the following two assumptions on the diffusion coefficient $a$ in
\eqref{eq:PDE}.
\begin{itemize}
  \item There exist two positive constants $0<a_{\mathrm{min}}<
a_{\mathrm{max}}<\infty$ such that for all $x\in \overline{D}$ and
$\mathbf{y}\in \Gamma$,
\begin{equation}\label{eq:esta}
a_{\mathrm{min}}\leq a(x,\mathbf{y}) \leq a_{\mathrm{max}}.
\end{equation}
  \item The complex continuation $a(x,\mathbf{z})$ of $a(x,\mathbf{y})$ is a
$L^{\infty}(D)$-valued holomorphic function on $\mathbb{C}^{d}$.
\end{itemize}
\begin{proposition}\label{prop:LegBounds}
Assume that the coefficient $a(x,\mathbf{y})$ in  the parameterized
PDEs \eqref{eq:PDE} satisfies the above two assumptions. If we
further require that $\Re(a(x,\vec{z}))\geq \delta$ for some
$0<\delta<a_{\mathrm{min}}$, $x\in \overline{D}$ and $\vec{z}\in
\mathcal{E}_{\pmb{\rho}}$ with $\rho_j>1$ for each $j=1,\ldots,d$.
Then, the coefficients of tensorized Legendre expansions of $u$
given in  \eqref{eq:tensorLeg} admit the following estimates.
\begin{align}\label{eq:estextLeg}
\|u_{\vec{k}}(x)\|_{V(D)} \leq \widehat{C}_{\pmb{\vec{\rho}},\delta}
\pmb{\vec{\rho}}^{-\vec{k}} \prod_{k_j\neq0}\sqrt{k_j}, \quad
\|v_{\vec{k}}(x)\|_{V(D)} \leq \widehat{C}_{\pmb{\vec{\rho}},\delta}
2^{\frac{\aleph(\mathbf{k})}{2}} \pmb{\vec{\rho}}^{-\vec{k}},
\end{align}
where
\[
\widehat{C}_{\pmb{\vec{\rho}},\delta}=\frac{\|f\|_{V^{*}(D)}}{\delta}
\left(\prod_{j=1}^{d}
\frac{L(\mathcal{E}_{\rho_j})}{\pi}\right)\prod_{k_i=0}\overline{D}_{\rho_i}^{(\frac{1}{2})}
\prod_{k_j\neq0} D_{\rho_j}^{(\frac{1}{2})}
\]
with $V^{\ast}(D)$ being the dual of the space $V(D)$.
\end{proposition}
\begin{proof}
By \cite[Theorem 1]{tran2017analysis}, it follows that the
conditions satisfied by $a(x,\mathbf{y})$ ensure that $\mathbf{z}
\rightarrow u(x,\mathbf{z})$ is analytic in an open neighborhood of
$\mathcal{E}_{\pmb{\rho}}$ and this solution satisfies a priori
estimate
\begin{equation}\label{eq:priest}
||u||_{V(D)}\leq \frac{||f||_{V^\ast(D)}}{\delta}.
\end{equation}
As a consequence, the solution $u$ satisfies the conditions of
Proposition \ref{prop:extLegcoeff}, and the estimates
\eqref{eq:estextLeg} follow directly from \eqref{eq:LegCoeffgen},
\eqref{eq:NormLegCoeffgen} and \eqref{eq:priest}.
\end{proof}

\begin{remark}
The following estimates of $\|u_{\vec{k}}(x)\|_{V(D)}$ and
$\|v_{\vec{k}}(x)\|_{V(D)}$ are reported in \cite[Proposition
2]{tran2017analysis}:
\begin{align}\label{eq:estTran}
\|u_{\vec{k}}(x)\|_{V(D)} \leq C_{\pmb{\vec{\rho}},\delta}
\pmb{\vec{\rho}}^{-\vec{k}} \prod_{j=1}^{d}(2k_j+1), \quad
\|v_{\vec{k}}(x)\|_{V(D)} \leq C_{\pmb{\vec{\rho}},\delta}
\pmb{\vec{\rho}}^{-\vec{k}} \prod_{j=1}^{d}\sqrt{2k_j+1},
\end{align}
where $C_{\pmb{\vec{\rho}},\delta}=\frac{\|f\|_{V^{*}(D)}}{\delta}
\prod_{j=1}^{d} \frac{L(\mathcal{E}_{\rho_j})}{4(\rho_j-1)}$. where
$C$ is some constant independent of the multi-index $\vec{k}$. A
comparison of \eqref{eq:estTran} and \eqref{eq:estextLeg} shows our
results \eqref{eq:estextLeg} are sharper, especially when $k_j \to
+\infty$ for some $j\in\{1,\ldots, d\}$.
\end{remark}

\begin{remark}
Since the polyellipse $\mathcal{E}_{\pmb{\rho}}$ could be deformed continuously to the hypercube $\Omega_d$ as $\pmb{\rho} \to \pmb{1}$, it is then reasonable to expect the uniform ellipticity of the diffusion coefficient $a(x,\mathbf{y})$ given in \eqref{eq:esta} also holds for some polyellipses $\mathcal{E}_{\pmb{\rho}}$, at least for $\pmb{\rho}$ close to $\pmb{1}$, which in turn implies the analyticity of the solution $u$ with respect to the parameter $\pmb{y}$ in a polyellipse. This explains the advance of using Legendre expansion (hence also for Gegenbauer expansion) over other kind of expansions (for example the Taylor expansion) in numerical studies of \eqref{eq:PDE}, as observed in \cite{tran2017analysis}.
\end{remark}

\section{Concluding remarks}\label{sec:conclusion}
In this paper, we have derived some new and sharper bounds for the
coefficients of multivariate Gegenbauer expansion of analytic functions based on two
different extensions of the Bernstein ellipse. These bounds
allow us to establish an explicit error bound for the multivariate
Gegenbauer approximation associated with an $\ell^q$ ball index set
in the uniform norm. For isotropic functions, the predicted rates of
convergence agree well with the empirical rates observed in the
numerical experiments. Moreover, our analysis suggests that the
multivariate Gegenbauer approximation based on the index set
$\Lambda_N^2$ is an optimal choice among that of the $\ell^q$ ball
index set, provided that the convergence rate established in Theorem
\ref{thm:BallSet} is sharp. Corresponding results for functions with finite regularity are also presented by restricting the discussion on a class of functions $\mathcal{H}^{\vec{m}}(\Omega_d)$
and on the full grid. As an application, we improve the estimates of the coefficients of tensorized
Legendre expansion arising from polynomial approximation for a
family of parameterized PDEs.

\section*{Acknowledgements}
We thank two anonymous referees for their careful reading and
constructive suggestions. Haiyong Wang also thanks the hospitality
of the School of Mathematical Sciences at Fudan university where the
present research was initiated. His work was supported by National
Natural Science Foundation of China under grant number 11671160. Lun
Zhang was partially supported by National Natural Science Foundation
of China under grant number 11822104, by The Program for Professor
of Special Appointment (Eastern Scholar) at Shanghai Institutions of
Higher Learning, and by Grant EZH1411513 from Fudan University.


\end{document}